\documentclass[11pt, reqno]{amsart}

\author[P.~Leonetti]{Paolo Leonetti}
\address{
Universit\'a degli Studi dell'Insubria, via Monte Generoso 71, 21100 Varese, Italy}
\email{leonetti.paolo@gmail.com}
\urladdr{\url{https://sites.google.com/site/leonettipaolo/}} 

\keywords{Rough convergence, rough family, ideal convergence, 
rough ideal cluster points, rough ideal limit points.}
\subjclass[2010]{Primary: 40A35. Secondary: 54A20.}

\title{On Rough Ideal Convergence}

\usepackage[T1]{fontenc}
\usepackage{amsmath}
\usepackage{amssymb}
\usepackage{amsthm}
\usepackage[left=3.5cm, right=3.5cm, paperheight=11.8in]{geometry}
\usepackage{hyperref}
\usepackage{fancyhdr}
\usepackage{enumitem}
\usepackage{comment}
\usepackage{nicefrac}
\usepackage{mathrsfs}
\usepackage{bm}
\usepackage{graphicx}
\usepackage[utf8]{inputenc}
\usepackage{cancel}
\usepackage{mathtools}

\AtBeginDocument{%
   \def\MR#1{}
}

\newtheorem{thm}{Theorem}[section]
\newtheorem{cor}[thm]{Corollary}
\newtheorem{lem}[thm]{Lemma}
\newtheorem{prop}[thm]{Proposition}

\theoremstyle{definition} 
\newtheorem{defi}[thm]{Definition}
\let\olddefi\defi
\renewcommand{\defi}{\olddefi\normalfont}
\newtheorem{example}[thm]{Example}
\let\oldexample\example
\renewcommand{\example}{\oldexample\normalfont}
\newtheorem{rmk}[thm]{Remark}
\let\oldrmk\rmk
\renewcommand{\rmk}{\oldrmk\normalfont}

\newtheorem{claim}{\textsc{Claim}}

\hypersetup{
    pdftitle={On rough ideal convergence},
    pdfauthor={Paolo Leonetti},
    pdfmenubar=false,
    pdffitwindow=true,
    pdfstartview=FitH,
    colorlinks=true,
    linkcolor=blue,
    citecolor=green,
    urlcolor=cyan
}

\providecommand{\MR}[1]{}

\providecommand{\MR}{\relax\ifhmode\unskip\space\fi MR }

\begin{document}

\begin{abstract}
We continue the study of ideal convergence for 
sequences $(x_n)$ with values in a topological space $X$ with respect to a family $\{F_\eta:\eta\in X\}$ of subsets of $X$ with $\eta\in F_\eta$, 
where each $F_\eta$ measures the allowed ``roughness'' of convergence toward $\eta$. 

More precisely, after introducing the corresponding notions of cluster and limit points, we prove several inclusion and invariance properties, discuss their structural properties, and give examples showing that the rough notions are genuinely different from the classical ideal ones. 
\end{abstract}

\maketitle
\thispagestyle{empty}


\section{Introduction}\label{sec:intro}

In this work, 
which can be thought of as a continuation of \cite{LeonettiJCA}, 
we study the structural properties of cluster points and limit points in the context of rough ideal convergence for sequences taking values in topological spaces. 
As remarked in \cite{MR5014948}, the theory of ideals has many applications in infinite combinatorics, 
it 
has some relevant potential for the study of the geometry of Banach spaces \cite{MR4124855, MR3436368}, Polish modules \cite{HuSolecki}, and it can be considered as a source of useful examples \cite{Miek}. Hopefully, the addition of \textquotedblleft a degree of freedom,\textquotedblright\, namely, the rough families below, will be relevant for the same motivations as the classical theory of ideal convergence. To be more precise, we start by recalling the necessary definitions and the notion of rough ideal convergence.

Let 
$\mathcal{I}\subseteq \mathcal{P}(\omega)$ be an ideal on the nonnegative integers $\omega$, that is, a family closed under subsets and finite unions. 
Unless otherwise stated, 
it is also assumed that the family of finite subsets of $\omega$, denoted by $\mathrm{Fin}$, is contained in $\mathcal{I}$ and that $\omega \notin \mathcal{I}$. We also write $\mathcal{I}^+:=\mathcal{P}(\omega)\setminus \mathcal{I}$  for its family of $\mathcal{I}$-positive sets. 
An ideal $\mathcal{I}$ is said to be a $P$\emph{-ideal} if for every increasing sequence $(S_j: j \in \omega)$ in $\mathcal{I}$ there exists $S \in \mathcal{I}$ such that $S_j\setminus S$ is finite for all $j \in \omega$. 
Identifying $\mathcal{P}(\omega)$ with the Cantor space $\{0,1\}^\omega$, it is possible to speak about the topological complexity of ideals; 
for instance, the family $\mathcal{Z}$ of asymptotic density zero sets, that is, $\mathcal{Z}:=\{S\subseteq \omega: |S\cap [0,n]|=o(n) \text{ as }n\to \infty\}$, is an $F_{\sigma\delta}$ $P$-ideal, while $\mathrm{Fin}$ is an $F_\sigma$ $P$-ideal. We refer to \cite{MR1711328} for an introduction to the theory of ideals. 

In the following definitions, 
$X$ is a given topological space with topology $\tau$, and $\mathscr{F}:=\{F_\eta: \eta \in X\}$ is a \emph{rough family} on $X$, that is, a collection of subsets of $X$ with the property that $\eta \in F_\eta$ for all $\eta \in X$. We denote the degenerate rough family by 
$$
\mathscr{F}^\natural:=\{\{\eta\}: \eta \in X\}.
$$
To keep the notation simple, we suppress the symbol of the topology $\tau$ whenever it is understood from the context. 

\begin{defi}\label{defi:mainroughness}
A sequence $\bm{x}=(x_n: n \in \omega)$ taking values in $X$ is said to be $\mathcal{I}$\emph{-convergent to} $\eta \in X$ \emph{with roughness} 
$\mathscr{F}$, shortened as 
$(\mathcal{I}, \mathscr{F})\text{-}\lim\nolimits_n x_n =\eta$, provided that $\{n \in \omega: x_n \notin U\} \in \mathcal{I}$ 
for all open sets $U\subseteq X$ such that $F_\eta \subseteq U$. We also write 
$$
\mathsf{Lim}_{\bm{x}}(\mathcal{I}, \mathscr{F}):=\left\{\eta \in X: (\mathcal{I}, \mathscr{F})\text{-}\lim\nolimits_n x_n =\eta\right\}.
$$
If $\mathcal{I}=\mathrm{Fin}$ we simply speak about $\mathscr{F}$-convergence. 
\end{defi}

Of course, ideal convergence with roughness $\mathscr{F}^\natural$ is nothing but classical ideal convergence. However, it is known that Definition \ref{defi:mainroughness} provides a new notion which is not included in the classical one, even modifying ideals and topology on the underlying set $X$, see \cite[Proposition 1.2]{LeonettiJCA}. We refer to \cite{LeonettiJCA} and references therein for an introduction and characterizations on rough ideal convergence. 

\begin{defi}\label{defi:mainroughnessstar}
A sequence $\bm{x}=(x_n: n \in \omega)$ taking values in $X$ is said to be $\mathcal{I}^\star$\emph{-convergent to} $\eta \in X$ \emph{with roughness} 
$\mathscr{F}$, shortened as 
$(\mathcal{I}^\star, \mathscr{F})\text{-}\lim\nolimits_n x_n =\eta$, provided that there exists a set $S \in \mathcal{I}$ such that the subsequence $(x_n: n \in \omega\setminus S)$ is $\mathscr{F}$-convergent to $\eta$. We also write 
$$
\mathsf{Lim}_{\bm{x}}(\mathcal{I}^\star, \mathscr{F}):=\left\{\eta \in X: (\mathcal{I}^\star, \mathscr{F})\text{-}\lim\nolimits_n x_n =\eta\right\}.
$$
\end{defi}
It is clear that $\mathsf{Lim}_{\bm{x}}(\mathcal{I}^\star, \mathscr{F})\subseteq \mathsf{Lim}_{\bm{x}}(\mathcal{I}, \mathscr{F})$. 
Indeed, this generalizes the fact that $\mathcal{I}^\star$-convergence (that is, $(\mathcal{I}^\star, \mathscr{F}^\natural)$-convergence) is stronger than $\mathcal{I}$-convergence (that is, $(\mathcal{I}, \mathscr{F}^\natural)$-convergence). In a metric space with a nonisolated point, it is folklore that $\mathcal{I}$-convergence and $\mathcal{I}^\star$-convergence coincide if and only if $\mathcal{I}$ is a $P$-ideal, see e.g. \cite[Theorem 3.2]{MR1844385}. In our first result we show, under certain restrictions on $\mathscr{F}$ (which include e.g. the case where each $F_\eta$ is finite) the analogous characterization holds for the equivalence of the two above notions. 
To this aim, recall that a metric space $X$ is said to have the $\textsc{UC}$\emph{-property} if nonempty disjoint closed sets are at a positive distance apart, 
that is, for all nonempty disjoint closed sets $F, F^{\prime} \subseteq X$, there exists $\varepsilon>0$ such that $d(x,x^\prime)> \varepsilon$ for all $x \in F$ and $x^\prime \in F^\prime$, where $d$ is the metric on $X$. For instance, all compact metric spaces $X$ have the $\textsc{UC}$-property while Euclidean spaces $\mathbb{R}^k$ do not have it; see  \cite{MR1165059, MR42109} and references therein. 
 \begin{prop}\label{prop:Pideal_equiv_rough_Istar}
Let $\mathcal{I}$ be an ideal on $\omega$. Let also $X$ be a metric space with the \textsc{UC}-property and $\mathscr{F}$ be a rough family of closed sets such that there exist $\eta\in X$ and a sequence $\bm{y}\in X^\omega$ for which 
$\lim_k y_k=\eta$ and $y_k \notin F_\eta$ for all $k \in \omega$. 

Then $\mathcal{I}$ is a $P$-ideal if and only if $\mathsf{Lim}_{\bm{x}}(\mathcal{I}^\star, \mathscr{F})=\mathsf{Lim}_{\bm{x}}(\mathcal{I}, \mathscr{F})$ for all $\bm{x} \in X^\omega$. 
\end{prop}

In Section \ref{sec:main} we introduce our main definitions of this work, namely, the notions of rough ideal cluster points and rough ideal limit points (Definition \ref{defi:clusterpointsrough} and Definition \ref{defi:limitpointsrough}, respectively), and we study several structural properties and relationships between them. The proofs of all our results are given in Section \ref{sec:proofs}. 

\section{Main Results}\label{sec:main}
\subsection{Rough ideal cluster points} 
We start with the notion of cluster point in the context of $(\mathcal{I}, \mathscr{F})$-convergence. 
\begin{defi}\label{defi:clusterpointsrough}
Let $\bm{x}=(x_n: n \in \omega)$ be a sequence taking values in $X$. Then $\eta \in X$ is said to be an $\mathcal{I}$\emph{-cluster point of} $\bm{x}$ \emph{with roughness} 
$\mathscr{F}$, 
or $(\mathcal{I}, \mathscr{F})$\emph{-cluster point}, 
if $\{n \in \omega: x_n \in U\} \notin \mathcal{I}$ 
for all open sets $U\subseteq X$ such that $F_\eta \subseteq U$. Then, we define
$$
\Gamma_{\bm{x}}(\mathcal{I}, \mathscr{F}):=\left\{\eta \in X: \eta \text{ is an }(\mathcal{I}, \mathscr{F})\text{-cluster point of }\bm{x}\right\}.
$$
In the case of degenerate rough family, we simply write $\Gamma_{\bm{x}}(\mathcal{I}):=\Gamma_{\bm{x}}(\mathcal{I}, \mathscr{F}^\natural)$.
\end{defi}

It is immediate that
$$
\mathsf{Lim}_{\bm{x}}(\mathcal{I}, \mathscr{F})\subseteq \Gamma_{\bm{x}}(\mathcal{I}, \mathscr{F}).
$$ 
(This corrects an unfortunate typo in the last centered formula before \cite[Theorem 1.3]{LeonettiJCA}.) 
Observe that a point $\eta \in X$ belongs to $\Gamma_{\bm{x}}(\mathcal{I})$ precisely when $\eta$ is an $\mathcal{I}$-cluster point of $\bm{x}$, see \cite{MR3920799} and references therein for characterizations of $\Gamma_{\bm{x}}(\mathcal{I})$. There is a connection between the sets $\mathsf{Lim}_{\bm{x}}(\mathcal{I}, \mathscr{F})$ and $\Gamma_{\bm{x}}(\mathcal{I})$: in fact, it has been proved in 
\cite[Corollary 1.4]{LeonettiJCA} 
that, if the sequence $\bm{x}$ has a relatively compact image, $X$ is a regular topological space, and 
each $F_\eta$ is closed, then 
$$
\mathsf{Lim}_{\bm{x}}(\mathcal{I}, \mathscr{F})
=\left\{\eta \in X: \Gamma_{\bm{x}}(\mathcal{I})\subseteq F_\eta \right\}.
$$
Below we present an analogous characterization for $(\mathcal{I}, \mathscr{F})$-cluster points (hereafter, we use $\overline{S}$ for the closure of $S\subseteq X$). 
\begin{thm}\label{thm:01thmclusterrough}
Let $\mathcal{I}$ be an ideal on $\omega$ and $\bm{x}$ be a sequence taking values in a regular topological space $X$ such that $\{n \in \omega: x_n \notin K\} \in \mathcal{I}$ for some compact $K\subseteq X$. Also, pick a rough family $\mathscr{F}$. Then  
\begin{equation}\label{eq:firstinclusionroughclusterpoints}
\left\{\eta \in X: F_\eta \cap \Gamma_{\bm{x}}(\mathcal{I})\neq \emptyset\right\}\subseteq 
\Gamma_{\bm{x}}(\mathcal{I}, \mathscr{F})\subseteq \left\{\eta \in X: \overline{F_\eta} \cap \Gamma_{\bm{x}}(\mathcal{I})\neq \emptyset\right\}.
\end{equation}

In particular, if each $F_\eta$ is closed, then 
$$
\Gamma_{\bm{x}}(\mathcal{I}, \mathscr{F})=\left\{\eta \in X: F_\eta \cap \Gamma_{\bm{x}}(\mathcal{I})\neq \emptyset\right\}.
$$
\end{thm}

It is easy to see that the above result depends crucially on the hypothesis $\{n \in \omega: x_n \notin K\} \in \mathcal{I}$ for some compact $K\subseteq X$. Indeed, consider the following example: $X=\mathbb{R}$, $\mathcal{I}$ is an arbitrary ideal, $x_n:=n$ for all $n \in \omega$, and the rough family $\mathscr{F}$ is defined by $F_0:=\omega$ and $F_\eta:=\{\eta\}$ for all nonzero $\eta$. In particular, all $F_\eta$ are closed and $\Gamma_{\bm{x}}(\mathcal{I})=\emptyset$. It follows that 
$\Gamma_{\bm{x}}(\mathcal{I}, \mathscr{F})=\{0\}$ while there are no $\eta \in \mathbb{R}$ for which $F_\eta \cap \Gamma_{\bm{x}}(\mathcal{I})\neq \emptyset$. 

\begin{rmk}\label{rmk:additionalclaims}
    As it will be clear from the proof of Theorem \ref{thm:01thmclusterrough}, the first inclusion in \eqref{eq:firstinclusionroughclusterpoints} holds even if $X$ is not regular and there is no compact set $K$ for which $\{n \in \omega: x_n \notin K\} \in \mathcal{I}$. In addition, under the hypotheses of Theorem \ref{thm:01thmclusterrough}, the same proof shows that, if a point $\eta \in X$ is given with $F_\eta$ closed, then $\eta$ is an $(\mathcal{I}, \mathscr{F})$-cluster point if and only if $F_\eta \cap \Gamma_{\bm{x}}(\mathcal{I})\neq \emptyset$. 
\end{rmk}

In the next example we show that, even in the case of bounded real sequences and $\mathcal{I}=\mathrm{Fin}$, both inclusions in \eqref{eq:firstinclusionroughclusterpoints} may be strict. 
\begin{example}\label{example:strinctinclusion}
Set $X=\mathbb{R}$ with the standard topology, $\mathcal{I}=\mathrm{Fin}$, and define the bounded real sequence $\bm{x}$ by $x_n:=2^{-n}$ for all $n \in \omega$. Define the rough family $\mathscr{F}$ as follows: $F_1:=\{2^{-n}: n \in \omega\}$, $F_{1/3}:=\{3^{-n-1}: n \in \omega\}$, and $F_{\eta}:=\{\eta\}$ for all $\eta \in \mathbb{R}\setminus \{\nicefrac{1}{3},1\}$. Since $\bm{x}$ is convergent, we get $\{0\}=\Gamma_{\bm{x}}(\mathcal{I})\subseteq \Gamma_{\bm{x}}(\mathcal{I}, \mathscr{F})$. Then the following hold: 
\begin{enumerate}[label={\rm (\roman{*})}]
\item $\left\{\eta \in X: F_\eta \cap \Gamma_{\bm{x}}(\mathcal{I})\neq \emptyset\right\}=\{0\}$. This is clear from the definition of $\mathscr{F}$.
\item $\Gamma_{\bm{x}}(\mathcal{I}, \mathscr{F})=\{0,1\}$. This follows by checking each case. First, as observed above, $0$ is an $(\mathcal{I}, \mathscr{F})$-cluster point. Second, $1$ is an $(\mathcal{I}, \mathscr{F})$-cluster point since $F_1$ is the range of $\bm{x}$. Third, $\nicefrac{1}{3}$ is not an $(\mathcal{I}, \mathscr{F})$-cluster point because there exists an open set $U$ containing $F_{1/3}$ such that $U \cap \{x_n: n \in \omega\}=\emptyset$. Lastly, if $\eta \in \mathbb{R}\setminus \{0,\nicefrac{1}{3},1\}$, then there exist only finitely many $n \in \omega$ such that $x_n \in (\eta-\varepsilon,\eta+\varepsilon)$, provided that $\varepsilon>0$ is sufficiently small. 
\item $\left\{\eta \in X: \overline{F_\eta} \cap \Gamma_{\bm{x}}(\mathcal{I})\neq \emptyset\right\}=\{0, \nicefrac{1}{3},1\}$. This follows because   $\overline{F_\eta}=F_\eta \cup \{0\}$ if $\eta \in \{\nicefrac{1}{3},1\}$, and $F_\eta$ is closed otherwise.
\end{enumerate} 
\end{example}

The following consequence is immediate from Theorem \ref{thm:01thmclusterrough}.
\begin{cor}\label{cor:nonempty}
With the same hypotheses of Theorem \ref{thm:01thmclusterrough}, suppose that each $F_\eta$ is closed. Then $\Gamma_{\bm{x}}(\mathcal{I})\neq \emptyset$ if and only if $\Gamma_{\bm{x}}(\mathcal{I}, \mathscr{F})\neq \emptyset$. 
\end{cor}

At this point, with a similar notation in \cite[Definition 1.2]{FKL2026}, let $\mathscr{C}_X(\mathcal{I}, \mathscr{F})$ be the family of sets of the type $\Gamma_{\bm{x}}(\mathcal{I}, \mathscr{F})$ for some sequence $\bm{x}$ taking values in $X$, that is, 
$$
\mathscr{C}_X(\mathcal{I}, \mathscr{F}):=\left\{F\subseteq X: F=\Gamma_{\bm{x}}(\mathcal{I}, \mathscr{F}) \text{ for some }\bm{x} \in X^\omega\right\}.
$$
\begin{cor}\label{cor:familiescluster}
    Let $\mathcal{I}$ be an ideal on $\omega$ for which there exists a sequence of pairwise disjoint $\mathcal{I}$\text{-}positive sets \textup{(}for instance, a Baire ideal or a nowhere maximal ideal\textup{)}. 
    Let $X$ be a compact metric space, and pick a rough family $\mathscr{F}$ made of closed sets. Then 
    $$
    \mathscr{C}_X(\mathcal{I}, \mathscr{F})=\left\{\left\{\eta \in X: F_\eta \cap C\neq \emptyset \right\}: C\text{ is a nonempty closed subset of }X\right\}.
    $$
\end{cor}
As an application, if $X$ is the real interval $[a,b]$, $\mathcal{I}$ is a Baire ideal, and $F_\eta:=[\eta-\varepsilon,\eta+\varepsilon]\cap X$ for all $\eta \in X$ and a given $\varepsilon> 0$, then $\mathscr{C}_X(\mathcal{I}, \mathscr{F})$ is the family of restrictions on $X$ of the nonempty (necessarily finite) unions of closed intervals $[x,y]\subseteq [a-\varepsilon,b+\varepsilon]$ with length at least $2\varepsilon$.

Now, we recall the basic fact that the set $\Gamma_{\bm{x}}(\mathcal{I})$ (that is, $\Gamma_{\bm{x}}(\mathcal{I}, \mathscr{F}^\natural)$) is always closed, see e.g. \cite[Lemma 3.1(iv)]{MR3920799}. One might wonder whether the same property holds replacing $\mathscr{F}^\natural$ with arbitrary rough families $\mathscr{F}$: for instance, is it true that elements of $\mathscr{C}_X(\mathcal{I}, \mathscr{F})$ in Corollary \ref{cor:familiescluster} are closed? In general, the answer is negative. Indeed, suppose that $X=\mathbb{R}$, $\mathcal{I}=\mathrm{Fin}$, and 
that $\bm{x}$ is the sequence defined by $x_n=(-1)^n$ for all $n \in \omega$. 
Choosing the rough family $\mathscr{F}=\{(\eta-3,\eta+3): \eta \in X\}$, it is easy to check that 
$\Gamma_{\bm{x}}(\mathcal{I}, \mathscr{F})=(-4,4)$. The answer remains negative even if we assume, in addition, that each member of the rough family is closed. In fact, consider the very same example, with the new rough family $\mathscr{F}^\prime=\{F_\eta^\prime: \eta\in X\}$ defined by $F^\prime_{\eta}:=\{1,\eta\}$ if $|\eta|<2$ and $F^\prime_{\eta}:=\{\eta\}$ otherwise. Then each $F^\prime_\eta$ is closed and $\Gamma_{\bm{x}}(\mathcal{I}, \mathscr{F}^\prime)=(-2,2)$.

To provide a sufficient condition for the closedness of $\Gamma_{\bm{x}}(\mathcal{I}, \mathscr{F})$, we recall the following notions. 
Given a topological space $X$, we endow the hyperspace 
$$
\mathcal{H}(X):=\{F\subseteq X: F \text{ nonempty closed}\}
$$
with the \emph{upper Vietoris topology} $\widehat{\tau}$, that is, the topology generated by the base of sets $\{F \in \mathcal{H}(X): F\subseteq U\}$, with $U \in \tau$ open. It is known that, if each $F_\eta$ is closed and the map $\eta\mapsto F_\eta$ is $\widehat{\tau}$-continuous, then $\mathsf{Lim}_{\bm{x}}(\mathcal{I}, \mathscr{F})$ is closed, see \cite[Theorem 1.6]{LeonettiJCA}. We show that the analogous claim holds also for the set of $(\mathcal{I}, \mathscr{F})$-cluster points. 

\begin{thm}\label{thm:clustersetclosed}
Let $\bm{x}$ be a sequence taking values in a topological space $X$, let $\mathcal{I}$ be an ideal on $\omega$, and pick a rough family $\mathscr{F}$ made of closed sets. Also, suppose that the map 
$\eta\mapsto F_\eta$ is $\widehat{\tau}$-continuous. 
Then $\Gamma_{\bm{x}}(\mathcal{I}, \mathscr{F})$ is closed. 
\end{thm}

For a practical application of Theorem \ref{thm:clustersetclosed}, denote by $B_r(\eta)$ the closed ball with center $\eta$ and radius $r \in [0,\infty]$. 
A metric space $(X,d)$ has the \emph{Heine--Borel property} if every closed bounded subset is compact. 

\begin{lem}\label{lem:hattaucontinuous}
    Let $X$ be a metric space with the Heine--Borel property  and pick an upper semicontinuous function $r: X\to [0,\infty)$. Then the map $X\to \mathcal{H}(X)$ given by $\eta \mapsto B_{r(\eta)}(\eta)$ is $\widehat{\tau}$-continuous. 
\end{lem}


As an immediate consequence of Theorem \ref{thm:clustersetclosed}, \cite[Theorem 1.6]{LeonettiJCA}, and Lemma \ref{lem:hattaucontinuous}, we get the following (in particular, this provides a correction to \cite[Corollary 1.7]{LeonettiJCA}, cf. Example \ref{ex:correcteJCA17} below).
\begin{cor}\label{cor:UCpropertycluster}
Let $\bm{x}$ be a sequence taking values in a metric space $X$ with the Heine--Borel property, let $\mathcal{I}$ be an ideal on $\omega$, and fix an upper semicontinuous function $r: X\to [0,\infty)$ such that $F_\eta=B_{r(\eta)}(\eta)$ for all $\eta \in X$.  

Then both $\mathsf{Lim}_{\bm{x}}(\mathcal{I}, \mathscr{F})$ and $\Gamma_{\bm{x}}(\mathcal{I}, \mathscr{F})$ are closed. 
\end{cor}

For $X=\mathbb R$, Theorem \ref{tmm:uppersemicrimpliesGammaLimcont} also proves continuity of $\bm{x}\mapsto\Gamma_{\bm{x}}(\mathcal I,\mathscr F)$ on $\ell_\infty$ when $r$ is bounded; cf. Example \ref{ex:negative}.

\begin{example}
Let $X:=\ell_\infty$ be the Banach space of bounded real sequences,
endowed with the supremum norm, and define $F_\eta:=B_1(\eta)$
for each $\eta\in \ell_\infty$. Consider the sequence 
\[
\bm{x}=(e_0,-e_0,e_1,-e_1,\ldots),
\]
where $e_k$ is the $k$th unit vector. Then it is routine to show that 
\[
\mathsf{Lim}_{\bm{x}}(\mathrm{Fin},\mathscr F)
=c_{00}\cap B_1(0)\quad \text{ and }\quad 
\Gamma_{\bm{x}}(\mathrm{Fin},\mathscr F)=B_1(0).
\]
Here $c_{00}$ denotes the subspace of 
eventually zero sequences (the first identity corrects a typo in the example following \cite[Theorem 1.9]{LeonettiJCA}). Note that $c_{00}\cap B_1(0)$ is not closed. However, it is folklore that the Banach space $\ell_\infty$ has neither the \textsc{UC}-property nor the Heine--Borel property. 
\end{example}

Below, we show that Corollary \ref{cor:UCpropertycluster} fails if the Heine--Borel property of the underlying metric space is replaced by the \textsc{UC}-property. 
\begin{example}\label{ex:correcteJCA17}
Set $Y:=[0,\nicefrac{1}{2}]$, and pick an injective sequence $(z_n: n \ge 1)$ with $z_n \notin Y$ for all $n$. Then, consider the metric space $(X,d)$, where $X:=Y\cup \{z_n: n \ge 1\}$ and the metric $d$ is defined, for all $a,b \in X$, by 
\begin{displaymath}
    d(a,b):=
    \begin{cases}
    \, 0,\,\,\,&\text{ if }a=b,\\
        \,|a-b| \,\,\,\,&\text{ if }a,b \in Y,\\
        \, 1+\nicefrac{1}{n}-a& \text{ if }a \in Y \text{ and }b=z_n \text{ for some }n\ge 1,\\
        \, 1+\nicefrac{1}{n}-b& \text{ if }b \in Y \text{ and }a=z_n \text{ for some }n\ge 1,\\
        \,1+\nicefrac{1}{n}+\nicefrac{1}{m}&\text{ if }a=z_n \text{ and }b=z_m \text{ for some distinct }n,m\ge 1.
    \end{cases}
\end{displaymath}
It is routine to show that $d$ is indeed a metric on $X$. 
In addition, $X$ has the \textsc{UC}-property. Indeed,
$d(z_n,a)>\nicefrac{1}{2}$ whenever $a\ne z_n$. If two nonempty disjoint
closed sets $A,B\subseteq X$ had distance zero, we could choose
$a_k\in A$ and $b_k\in B$ with $\lim_kd(a_k,b_k)=0$.
Then $a_k,b_k\in Y$ for all sufficiently large $k$.
By compactness of $Y$, a subsequence of $(a_k)$ converges to
some $y\in Y$, and the corresponding subsequence of $(b_k)$
converges to the same point. Then, closedness would imply $y\in A\cap B$,
which is a contradiction.

Now, consider the rough family $\mathscr{F}$ defined by $F_\eta:=B_1(\eta)$ for all $\eta \in X$ (hence, $r\equiv 1$). Notice that $B_1(y)=Y\cup \{z_n: \nicefrac{1}{n}\le y\}$ for all $y\in Y$ and $B_1(z_n)=\{z_n\}\cup \{y\in Y:\nicefrac1n\le y\}$ for all $n\ge 1$. Consider the sequence $$
\bm{x}:=(z_1,z_2,z_3,\ldots).
$$ 
For every $y\in(0,\nicefrac12]$, $F_y=B_1(y)$ contains
all but finitely many terms of $\bm{x}$, hence
\[
(0,\nicefrac12]\subseteq
\mathsf{Lim}_{\bm{x}}(\mathrm{Fin},\mathscr F)
\subseteq\Gamma_{\bm{x}}(\mathrm{Fin},\mathscr F).
\]
On the other hand, $Y$ is open in $X$, contains $F_0=Y$,
and contains no term of $\bm{x}$. Thus $0$ is not a $(\mathrm{Fin}, \mathscr{F})$-cluster point. For each $m\ge1$, the set
$U_m:=Y\cup\{z_m\}$ is open, contains $F_{z_m}$, and contains
exactly one term of $\bm{x}$. Hence no $z_m$ is a
$(\mathrm{Fin}, \mathscr{F})$-cluster point. This implies that $\Gamma_{\bm{x}}(\mathrm{Fin},\mathscr F)\subseteq (0,\nicefrac{1}{2}]$. We conclude that 
\[
\mathsf{Lim}_{\bm{x}}(\mathrm{Fin},\mathscr F)
=\Gamma_{\bm{x}}(\mathrm{Fin},\mathscr F)
=(0,\nicefrac12].
\]
In particular, these sets are not closed. 
\end{example}


\subsection{Rough ideal limit points} 
We continue with the notion of limit point in the context of $(\mathcal{I}, \mathscr{F})$-convergence. 
\begin{defi}\label{defi:limitpointsrough}
Let $\bm{x}=(x_n: n \in \omega)$ be a sequence with values in $X$. Then $\eta \in X$ is said to be 
an $\mathcal{I}$\emph{-limit point of} $\bm{x}$ \emph{with roughness} $\mathscr{F}$, 
or $(\mathcal{I}, \mathscr{F})$\emph{-limit point}, 
if there exists a set $S \in \mathcal{I}^+$ such that the subsequence $(x_n: n \in S)$ is $\mathscr{F}$-convergent to $\eta$. 
Then, we define
$$
\Lambda_{\bm{x}}(\mathcal{I}, \mathscr{F}):=\left\{\eta \in X: \eta \text{ is an }(\mathcal{I}, \mathscr{F})\text{-limit point of }\bm{x}\right\}.
$$
In the case of degenerate rough family, we simply write $\Lambda_{\bm{x}}(\mathcal{I}):=\Lambda_{\bm{x}}(\mathcal{I}, \mathscr{F}^\natural)$.
\end{defi}

Observe that a point $\eta \in X$ belongs to $\Lambda_{\bm{x}}(\mathcal{I})$ precisely when $\eta$ is an $\mathcal{I}$-limit point of $\bm{x}$, see e.g. \cite{ 
MR2181783, 
MR1838788, 
MR1844385}. It is also easy to check that 
\begin{equation}\label{eq:trivialinclusion1}
\mathsf{Lim}_{\bm{x}}(\mathcal{I}^\star, \mathscr{F})
\subseteq 
\Lambda_{\bm{x}}(\mathcal{I}, \mathscr{F})
\subseteq \Gamma_{\bm{x}}(\mathcal{I}, \mathscr{F}).
\end{equation}
In addition, as remarked before, we have  
\begin{equation}\label{eq:trivialinclusion2}
\mathsf{Lim}_{\bm{x}}(\mathcal{I}^\star, \mathscr{F})
\subseteq 
\mathsf{Lim}_{\bm{x}}(\mathcal{I}, \mathscr{F})
\subseteq \Gamma_{\bm{x}}(\mathcal{I}, \mathscr{F}).
\end{equation}
\begin{rmk}
There is no immediate relationship between $\mathsf{Lim}_{\bm{x}}(\mathcal{I}, \mathscr{F})$ and $\Lambda_{\bm{x}}(\mathcal{I}, \mathscr{F})$, even if $X$ is compact Hausdorff and $\mathscr{F}=\mathscr{F}^\natural$. In fact, every real bounded nonconvergent sequence $\bm{x}$ satisfies $\Lambda_{\bm{x}}(\mathrm{Fin}, \mathscr{F}^\natural)\setminus \mathsf{Lim}_{\bm{x}}(\mathrm{Fin}, \mathscr{F}^\natural)\neq \emptyset$. 

Vice versa, suppose that $X:=\beta \omega$ is the \v{C}ech--Stone compactification of $\omega$, which can be identified as the space of principal and free ultrafilters on $\omega$ and recall that a basic clopen set of $X$ is of the type $\hat{A}:=\{q \in \beta\omega: A \in q\}$ for some $A\subseteq \omega$. Fix a maximal ideal $\mathcal{I}$ on $\omega$, so that $p:=\mathcal{I}^+$ is a free ultrafilter on $\omega$. Define the sequence $\bm{x}$ with values in $X$ such that $x_n$ is the principal ultrafilter at $n$, for each $n \in \omega$. Then $\bm{x}$ is $\mathcal{I}$-convergent to $p$: indeed, if $U$ is an arbitrary open neighborhood of $p$, then $\hat{A}\subseteq U$ for some $A \in p$, which implies that 
$$
\{n \in \omega: x_n \notin U\}\subseteq \{n \in \omega: x_n \notin \hat{A}\}=\omega\setminus A \in \mathcal{I}.
$$
On the other hand, the subspace $\omega\subseteq X$ is discrete and sequentially closed; hence, there are no subsequences of $\bm{x}$ which converge to the free ultrafilter $p$. This proves that $p$ is not an $\mathcal{I}$-limit point of $\bm{x}$. Therefore $\mathsf{Lim}_{\bm{x}}(\mathcal{I}, \mathscr{F}^\natural)\setminus \Lambda_{\bm{x}}(\mathcal{I}, \mathscr{F}^\natural)\neq \emptyset$. 
\end{rmk}

\begin{prop}\label{prop:basicIlimitpoints}
Let $\mathcal{I}$ be an ideal on $\omega$,  
$\bm{x}$ be a sequence taking values in a topological space $X$, and pick a rough family $\mathscr{F}$. Then 
$$
\left\{\eta \in X: F_\eta \cap \Lambda_{\bm{x}}(\mathcal{I})\neq \emptyset\right\}\subseteq 
\Lambda_{\bm{x}}(\mathcal{I}, \mathscr{F}).
$$
\end{prop}

It is worth remarking that, differently from the case of $(\mathcal{I}, \mathscr{F})$-cluster points in Theorem \ref{thm:01thmclusterrough}, here it is hopeless to try to obtain an inclusion of the type $$\Lambda_{\bm{x}}(\mathcal{I}, \mathscr{F})\subseteq \left\{\eta \in X: \tilde{F}_\eta \cap \Lambda_{\bm{x}}(\mathcal{I})\neq \emptyset\right\},$$ where $\tilde{F}_\eta$ is a set depending on $F_\eta$. Indeed, as follows from the next example, it is possible (also in compact metric spaces) that $\Lambda_{\bm{x}}(\mathcal{I})$ is empty while $\Lambda_{\bm{x}}(\mathcal{I}, \mathscr{F})$ is not. This goes in the opposite direction of 
Corollary \ref{cor:nonempty}. 
\begin{example}
Set $X=\mathbb{R}$ with the standard topology, and let $\bm{x}$ be a real sequence which is equidistributed in $[0,1]$; for instance, one may consider Fridy's sequence in \cite[Example 4]{MR1181163} or, more generally, the sequences constructed in the proof of \cite[Theorem 3.1]{MR3883171}. Let also $\mathscr{F}$ be a rough family such that $\eta$ is an interior point of $F_\eta$, for each $\eta \in X$; for instance, $F_\eta$ might be $B_{r(\eta)}(\eta)$ with $r$ being a strictly positive map. Choosing the ideal $\mathcal{Z}$ of asymptotic density zero sets, it easily follows that
$$
[0,1]\subseteq \Lambda_{\bm{x}}(\mathcal{Z}, \mathscr{F})
\quad \text{ and }\quad 
\Lambda_{\bm{x}}(\mathcal{Z})=\emptyset;
$$
cf. the argument contained in \cite[Example 4]{MR1181163}.
\end{example}

Lastly, the connection between $(\mathcal{I}, \mathscr{F})$-limit points and classical $\mathcal{I}$-limit points is given in the next result. To this aim, given a sequence $\bm{x}=(x_n: n \in \omega)$ with values in a metric space $X$ and a nonempty subset $Y \subseteq X$, we define the real sequence 
$$
d(\bm{x}, Y):=\left(d(x_n,Y): n \in \omega\right),
$$
where $d$ is the metric on $X$ and $d(x_n,Y):=\inf\{d(x_n,y): y \in Y\}$ for each $n \in \omega$. 
\begin{thm}\label{thm:ilimitpointcharact}
Let $\mathcal{I}$ be an ideal on $\omega$, $X$ be a metric space with the $\mathsf{UC}$-property, and fix a point $\eta \in X$. Pick also a rough family $\mathscr{F}$ such that $F_\eta$ is closed. Then 
$$
\forall \bm{x} \in X^\omega, \quad \quad 
\eta \in \Lambda_{\bm{x}}(\mathcal{I}, \mathscr{F}) 
\quad \text{ if and only if }\quad 
0 \in \Lambda_{d(\bm{x}, F_\eta)}(\mathcal{I}).
$$
\end{thm}

It turns out that the latter equivalence characterizes the \textsf{UC}-property of $X$, as we show in Remark \ref{rmk:ICproperytyequivalence} below.
\begin{rmk}\label{rmk:xxxx}
        In the special case where $X=\mathbb{R}$ (which does not have the \textsc{UC}-property) and there exists a radius $r \in [0,\infty)$ such that $F_\eta=B_r(\eta)$ for all $\eta \in \mathbb{R}$, the analogous statements of Proposition \ref{prop:Pideal_equiv_rough_Istar} and Theorem \ref{thm:ilimitpointcharact} hold as well. In fact, in each case, it is enough to observe that if $U$ is open set containing $F_\eta$ then there exists $k \in \omega$ such that $F_\eta\subseteq \{x \in \mathbb{R}: d(x,F_\eta)<2^{-k}\}\subseteq U$. 
\end{rmk}

\subsection{Relationships} 
In this section, we show some relationships among the four main notions of this work. Of course, we already observed the trivial inclusions \eqref{eq:trivialinclusion1} and \eqref{eq:trivialinclusion2}. 

Now, we show that the set $\mathsf{Lim}_{\bm{x}}(\mathcal{I}^\star, \mathscr{F})$ can be much smaller than all the other ones.
\begin{prop}\label{prop:existence}
    There exist a sequence $\bm{x}$ with values in a compact metric space $X$, a rough family $\mathscr{F}$ made of closed sets, and an ideal $\mathcal{I}$ on $\omega$ such that 
    $$
    \mathsf{Lim}_{\bm{x}}(\mathcal{I}^\star, \mathscr{F})=\emptyset 
    \quad \text{ and }\quad \mathsf{Lim}_{\bm{x}}(\mathcal{I}, \mathscr{F})=\Gamma_{\bm{x}}(\mathcal{I}, \mathscr{F})=\Lambda_{\bm{x}}(\mathcal{I}, \mathscr{F})=X.
    $$
\end{prop}
It turns out that the construction in the proof of Proposition \ref{prop:existence} satisfies also two additional properties: the map $\eta\to F_\eta$ is $\widehat{\tau}$ continuous, and there exist a point $\eta \in X$ and a sequence $\bm{y} \in X^\omega$ such that $\lim_k y_k=\eta$ and $y_k \notin F_\eta$ for all $k \in \omega$. Hence, as follows by Proposition \ref{prop:Pideal_equiv_rough_Istar}, $\mathcal{I}$ is necessarily not a $P$-ideal (in fact, it will be a copy on $\omega$ of the Fubini product $\mathrm{Fin}\times \mathrm{Fin}$); see Remark \ref{rmk:hattaucontinuous} below. 

It is known, in the classical case where $\mathscr{F}=\mathscr{F}^\natural$, that the vector space of bounded $\mathcal{I}$-convergent real sequences coincides with the closure (with respect to the supremum norm in the Banach space $\ell_\infty$) of the subspace of bounded $\mathcal{I}^\star$-convergent real sequences, see \cite[Theorem 2.4]{MR2181783}. In the following, we provide a positive extension of the latter result in the context of $(\mathcal{I}, \mathscr{F})$-convergence. 
To this aim, 
we write 
$$
c^b(\mathcal{I}, \mathscr{F}):=\left\{\bm{x} \in 
\ell_\infty:  
\mathsf{Lim}_{\bm{x}}(\mathcal{I}, \mathscr{F})\neq \emptyset\right\}, 
$$
adapting the notation from \cite[Section 1]{MR4947872}. 
It is worth remarking that, if $\mathcal I$ is nonmaximal and $\sup_{\eta\in\mathbb R}\operatorname{diam}(F_\eta)<\infty$, then $c^b(\mathcal I,\mathscr F)$ is a vector space if and only if $\mathscr{F}$ is degenerate; see \cite[Proposition 1.5]{LeonettiJCA} and cf. item \ref{item:4corr} in the note added in proof.
Lastly, to state our next result, we define also 
$$
c^b(\mathcal{I}^\star, \mathscr{F}):=\left\{\bm{x} \in \ell_\infty: 
\mathsf{Lim}_{\bm{x}}(\mathcal{I}^\star, \mathscr{F})\neq \emptyset\right\}. 
$$
Of course, $c\subseteq c^b(\mathcal{I}^\star, \mathscr{F})\subseteq c^b(\mathcal{I}, \mathscr{F})\subseteq \ell_\infty$ and, under mild assumptions, 
the equality $c^b(\mathcal{I}^\star, \mathscr{F})= c^b(\mathcal{I}, \mathscr{F})$  holds 
if $\mathcal{I}$ is a $P$-ideal, cf. Proposition \ref{prop:Pideal_equiv_rough_Istar} and Remark \ref{rmk:xxxx}. 
\begin{thm}\label{thm:closure}
Set $X=\mathbb{R}$. Let $\mathcal{I}$ be an ideal on $\omega$, and pick a rough family $\mathscr{F}$ of closed sets such that
 $\sup_{\eta\in\mathbb R}\operatorname{diam}(F_\eta)<\infty$. Suppose also that $\eta\mapsto F_\eta$ is $\widehat\tau$-continuous. 
Then $c^b(\mathcal{I}, \mathscr{F})$ coincides with the closure of $c^b(\mathcal{I}^\star, \mathscr{F})$. 
\end{thm}

Now, we show that, under certain restrictions, both maps $\bm{x}\mapsto \mathsf{Lim}_{\bm{x}}(\mathcal{I}, \mathscr{F})$ and $\bm{x}\mapsto \Gamma_{\bm{x}}(\mathcal{I}, \mathscr{F})$ are continuous on $\ell_\infty$.
\begin{thm}\label{tmm:uppersemicrimpliesGammaLimcont}
Set $X=\mathbb{R}$. Let $\mathcal{I}$ be an ideal on $\omega$ and fix a bounded and upper semicontinuous function $r: X\to [0,\infty)$ such that $F_\eta=B_{r(\eta)}(\eta)$ for all $\eta \in X$. Then the maps $\Phi: c^b(\mathcal{I}, \mathscr{F}) \to \mathcal{H}(X)$ and $\Psi: \ell_\infty \to \mathcal{H}(X)$ defined by 
$$
\Phi(\bm{x}):=\mathsf{Lim}_{\bm{x}}(\mathcal{I}, \mathscr{F}) 
\quad \text{ and }\quad 
\Psi(\bm{x}):=\Gamma_{\bm{x}}(\mathcal{I}, \mathscr{F}) 
$$
are both $\widehat{\tau}$-continuous. 
\end{thm}
Both maps are well defined, with nonempty closed values. On $\mathbb R$, upper semicontinuity of $r$ already implies closedness of the rough limit and cluster sets without boundedness of $r$, by Lemma \ref{lem:hattaucontinuous}, \cite[Theorem 1.6]{LeonettiJCA}, and Corollary \ref{cor:UCpropertycluster}. The next example shows that boundedness cannot be omitted from the continuity theorem, even when $r$ is continuous and $\mathcal I=\mathrm{Fin}$.
\begin{example}\label{ex:negative}
Suppose that $X=\mathbb{R}$, $\mathcal{I}=\mathrm{Fin}$, and the rough family $\mathscr{F}$ is given as in Theorem \ref{tmm:uppersemicrimpliesGammaLimcont} once we consider the continuous function $r: \mathbb{R}\to [0,\infty)$ defined by
\[
\forall \eta \in \mathbb{R}, \quad 
r(\eta)=
\begin{cases}
\,0 \,\,\,& \text{ if }|\eta|\le 1,\\[2mm]
\,|\eta|-\dfrac1{|\eta|} & \text{ if } |\eta|\ge 1.
\end{cases}
\]
Let also $\bm{x}$ be the real constant sequence $0$. Using \cite[Corollary 1.4]{LeonettiJCA} and Theorem \ref{thm:01thmclusterrough}, it easily follows that 
$\Phi(\bm{x})=\mathsf{Lim}_{\bm{x}}(\mathcal{I},\mathscr{F})=\{0\}$ and $\Psi(\bm{x})=\Gamma_{\bm{x}}(\mathcal{I},\mathscr{F})=\{0\}$. 
To conclude, we claim that both maps $\Phi$ and $\Psi$ are \emph{not} $\widehat{\tau}$-continuous at $\bm{x}$.

Now fix the open set $U:=(-1,1)$. 
Then $\{0\}\subseteq U$, i.e., $\{0\}\in \widehat{U}:=\{F \in \mathcal{H}(\mathbb{R}): F\subseteq U\}$. 
It will be enough to show that, for every $\varepsilon>0$, there exists $\bm{y}\in c^b(\mathcal{I}, \mathscr{F})$ such that $\|\bm{x}-\bm{y}\|_\infty<\varepsilon$ and both $\mathsf{Lim}_{\bm{y}}(\mathcal{I},\mathscr{F})$ and $\Gamma_{\bm{y}}(\mathcal{I},\mathscr{F})$ are not contained in $U$. 
To this aim, fix $\varepsilon\in (0,1)$ and consider the sequence $\bm{y}$ which is constantly equal to $\nicefrac{\varepsilon}{2}$. 
Define $\eta:=\nicefrac{2}{\varepsilon}>2$, so that $F_{\eta}=[1/\eta, 2\eta-\nicefrac{1}{\eta}]$. 
Of course, $\|\bm{x}-\bm{y}\|_\infty<\varepsilon$ and $\Gamma_{\bm{y}}(\mathcal{I})=\{\nicefrac{\varepsilon}{2}\}=\{1/\eta\}$. 
Again by \cite[Corollary 1.4]{LeonettiJCA} and Theorem \ref{thm:01thmclusterrough}, we obtain that $\eta$ (which is not in $U$) belongs to both $\mathsf{Lim}_{\bm{y}}(\mathcal{I},\mathscr{F})$ and $\Gamma_{\bm{y}}(\mathcal{I},\mathscr{F})$. 
\end{example}

Next, we characterize the family of ideals for which the notions of $(\mathcal{I},\mathscr{F})$-cluster points and $(\mathcal{I},\mathscr{F})$-limit points coincide for every sequence. In the special case where $\mathscr{F}$ is the degenerate family $\mathscr{F}^\natural$ and $X$ is a nondiscrete metric space, it is known that the equivalence holds whenever $\mathcal{I}$ is an $F_\sigma$-ideal, see \cite[Theorem 2.3]{MR3883171} or \cite[Proposition 5.4]{MR4967065}. In addition, if $\mathcal{I}$ is an analytic $P$-ideal, then the equivalence holds if and only if $\mathcal{I}$ is $F_\sigma$, see \cite[Theorem 2.5]{MR3883171}. A definitive answer for the degenerate case has been obtained in \cite[Theorem 3.4]{MR4393937} and
\cite[Proposition 3.10]{FKL2026}: on a nondiscrete metric space, $\mathcal{I}$-cluster points and $\mathcal{I}$-limit points coincide for every sequence if and only if $\mathcal{I}$ is a $P^+$-ideal. 

Here, we recall that an ideal $\mathcal{I}$ on $\omega$ is said to be $P^+$\emph{-ideal} if for every decreasing sequence $(A_n:n \in \omega)$ of $\mathcal{I}$-positive sets there exists $A\in \mathcal{I}^+$ such that $A\setminus A_n$ is finite for every $n \in \omega$. It is known that all $G_{\delta\sigma}$-ideals are $P^+$, see \cite[Corollary 2.7]{MR5030363} or \cite[Proposition 3.2]{FKL2026}; cf. also \cite[Lemma 1.2]{MR748847} for the original proof in the case of $F_\sigma$-ideals. Remarkably, $\mathcal{Z}$ is not a $P^+$-ideal. 
Below we obtain the analogous characterization of \cite[Theorem 3.4]{MR4393937} in the context of $(\mathcal{I},\mathscr{F})$-convergence.

\begin{thm}\label{thm:characterizationPplus}
Let $\mathcal{I}$ be an ideal on $\omega$. Let $X$ be a compact metric space and $\mathscr{F}$ 
be a rough family of closed sets such that there exist $\eta\in X$ and a sequence $\bm{y} \in X^\omega$ for which 
$\lim_k y_k=\eta$ and $y_k\notin F_{\eta}$ for all $k\in\omega$. 

Then 
$\mathcal{I}$ is a $P^+$-ideal if and only if $\Lambda_{\bm{x}}(\mathcal{I},\mathscr{F})=\Gamma_{\bm{x}}(\mathcal{I},\mathscr{F})$ for every $\bm{x}\in X^\omega$\textup{.}
\end{thm}

Finally, we cannot remove the hypothesis of the existence of $\eta$ and $\bm{y}$: in fact, if $F_\lambda=X$ for all $\lambda \in X$, then $\Lambda_{\bm{x}}(\mathcal{I},\mathscr{F})=\Gamma_{\bm{x}}(\mathcal{I},\mathscr{F})=X$ for every $\bm{x}\in X^\omega$ and every ideal $\mathcal{I}$.

\section{Proofs}\label{sec:proofs}

\begin{proof}
[Proof of Proposition \ref{prop:Pideal_equiv_rough_Istar}]
\textsc{If part.} 
Let $(I_k: k \in \omega)$ be an increasing sequence of sets in $\mathcal{I}$ and set $I_\infty:=\bigcup_k I_k$. We need to show that there exists a set $I \in \mathcal{I}$ such that $I_k\setminus I$ is finite for all $k \in \omega$. If $I_\infty \in \mathcal{I}$ then it is enough to choose $I=I_\infty$. Hence, suppose hereafter that $I_\infty \in \mathcal{I}^+$. 

By hypothesis, it is possible to pick a sequence $\bm{y}\in X^\omega$ converging to some $\eta\in X$ such that $y_k \notin F_\eta$ for all $k \in \omega$. 
Passing to a subsequence if needed, we may assume without loss of generality that $\bm{y}$ is injective. 
Define the sequence $\bm{z}\in X^\omega$ by setting
$$
\forall k\in\omega,\ \forall n\in I_k\setminus I_{k-1},\qquad z_n:=y_k,
\qquad\text{and}\qquad
z_n:=\eta\ \text{ for all }n\in \omega\setminus I_\infty, 
$$
with the convention that $I_{-1}:=\emptyset$. 
We claim that $(\mathcal{I},\mathscr{F})\textup{-}\lim_n z_n=\eta$.
Indeed, let $U\subseteq X$ be an open set containing $F_\eta$.
Since $\lim_k y_k=\eta$, the set $F:=\{k\in\omega: y_k\notin U\}$ is finite.
Hence $\{n\in\omega: z_n\notin U\}\subseteq I_{\max(F\cup \{0\})} \in\mathcal{I}$, 
which proves the claim. 
We get by the standing hypothesis that
$$
\eta\in \mathsf{Lim}_{\bm{z}}(\mathcal{I}, \mathscr{F})=\mathsf{Lim}_{\bm{z}}(\mathcal{I}^\star, \mathscr{F}).
$$
Therefore there exists $S\in\mathcal{I}$ such that the subsequence $(z_n:n\in \omega\setminus S)$ is $\mathscr{F}$-convergent to $\eta$. 
Let us show that $I_k\setminus S\in\mathrm{Fin}$ for every $k\in\omega$. To this aim, fix $k \in \omega$. 
Since $y_k\notin F_\eta$ and $X$ is metric and $F_\eta$ is closed, there exists $\varepsilon_k>0$ such that
$B_{\varepsilon_k}(y_k)\cap F_\eta=\emptyset$ (recall that $B_{\varepsilon_k}(y_k)$ is the closed ball with center $y_k$ and radius $\varepsilon_k$). 
Define
$
U_k:=X\setminus B_{\varepsilon_k}(y_k).
$ 
Then $U_k$ is an open set containing $F_\eta$. Hence 
$$
(I_k\setminus I_{k-1})\setminus S\subseteq \{n\in \omega\setminus S: z_n\in B_{\varepsilon_k}(y_k)\}=\{n\in \omega\setminus S: z_n\notin U_k\}\in\mathrm{Fin}.
$$
It follows that $I_k\setminus S=\bigcup_{j\le k}(I_j\setminus I_{j-1})\setminus S$ is finite. 

Then $I:=\bigcup_k (I_k\cap S)\subseteq S\in\mathcal{I}$ and, in addition, 
$$
\forall k \in\omega, \quad 
I_k\setminus I
\subseteq I_k \setminus (I_k\cap S)
=I_k \cap (\omega\setminus S)=I_k\setminus S \in \mathrm{Fin}.
$$ Therefore $\mathcal{I}$ is a $P$-ideal. 

\medskip

\textsc{Only If part.} 
Suppose that $\mathcal{I}$ is a $P$-ideal and fix a sequence $\bm{x} \in X^\omega$. We only need to show that $\mathsf{Lim}_{\bm{x}}(\mathcal{I}, \mathscr{F})\subseteq \mathsf{Lim}_{\bm{x}}(\mathcal{I}^\star, \mathscr{F})$ (as the converse is obvious). This is clear if $\mathsf{Lim}_{\bm{x}}(\mathcal{I}, \mathscr{F})=\emptyset$. Otherwise, pick $\lambda \in X$ such that $(\mathcal{I},\mathscr{F})\textup{-}\lim \bm{x}=\lambda$. We need to show that $(\mathcal{I}^\star,\mathscr{F})\textup{-}\lim \bm{x}=\lambda$ holds as well. 

Let $d$ denote the given metric on $X$ and define the open sets 
$$
\forall j \in \omega, \quad 
U_j:=\left\{x \in X: \inf_{y \in F_\lambda}d(x,y)<2^{-j}\right\}.
$$
Of course, $F_\lambda \subseteq U_j$ for each $j \in \omega$. In addition, if $U\subseteq X$ is an open set containing $F_\lambda$, then $X\setminus U$ and $F_\lambda$ are disjoint closed sets, hence by the \textsc{UC}-property of $X$ they are at a positive distance apart; thus, 
there exists some $j \in \omega$ such that 
$F_\lambda \subseteq U_j \subseteq U$. 

At this point, define $A_j:=\{n \in \omega: x_n \notin U_j\}$ for all $j \in \omega$. Since $\lambda \in \mathsf{Lim}_{\bm{x}}(\mathcal{I}, \mathscr{F})$, it follows that $(A_j: j \in \omega)$ is an increasing sequence of sets in $\mathcal{I}$. By the $P$-property of $\mathcal{I}$, there exists $A \in \mathcal{I}$ such that $A_j\setminus A$ is finite for all $j \in \omega$. To conclude the proof, it is enough to prove that the subsequence $(x_n: n \in \omega\setminus A)$ is $\mathscr{F}$-convergent to $\lambda$. In fact, if $U$ is an open set containing $F_\lambda$, we can pick an integer $j \in \omega$ with $F_\lambda\subseteq U_j\subseteq U$, hence 
$$
\{n\in \omega\setminus A: x_n\notin U\}\subseteq \{n\in \omega\setminus A: x_n\notin U_j\}
= A_j\setminus A\in\mathrm{Fin}.
$$
Therefore $(\mathcal{I}^\star,\mathscr{F})\textup{-}\lim_n x_n=\lambda$.
\end{proof}

\begin{rmk}
    As it is clear from the proof above, the \textsc{If part} works in every metric space (hence, not necessarily with the \textsc{UC}-property) and for every rough family $\mathscr{F}$ with at least $F_\eta$ closed (hence, not necessarily all its elements are closed). In addition, the \textsc{Only If part} works independently of the existence of $\eta\in X$ and $\bm{y}\in X^\omega$ such that $\lim_k y_k=\eta$ and $y_k \notin F_\eta$ for all $k \in \omega$. 
\end{rmk}

\medskip

\begin{proof}
    [Proof of Theorem \ref{thm:01thmclusterrough}] 
The first inclusion in \eqref{eq:firstinclusionroughclusterpoints} is obvious if there are no $\mathcal{I}$-cluster points. Otherwise, fix points $\eta,\eta^\prime \in X$ such that $\eta^\prime \in F_\eta \cap \Gamma_{\bm{x}}(\mathcal{I})$. Pick an open set $U\subseteq X$ containing $F_\eta$. Since $\eta^\prime \in U$, it follows that $\{n\in \omega: x_n \in U\}\in \mathcal{I}^+$. Hence, by the arbitrariness of $U$, we get $\eta \in \Gamma_{\bm{x}}(\mathcal{I}, \mathscr{F})$. 

\bigskip

To prove the second inclusion, let
$\eta\in\Gamma_{\bm{x}}(\mathcal I,\mathscr F)$ and put
$C=\overline{F_\eta}$. Suppose, towards a contradiction, that
$C\cap\Gamma_{\bm{x}}(\mathcal I)=\emptyset$.
The set $C\cap K$ is compact. For each $c\in C\cap K$, choose an open
neighborhood $V_c$ such that $\{n:x_n\in V_c\}\in\mathcal I$.
A finite subcover gives an open set $V\supseteq C\cap K$ with $\{n:x_n\in V\}\in\mathcal I$; take $V=\emptyset$ if $C\cap K$ is empty. The compact set $K\setminus V$ is disjoint from the closed set $C$. Regularity separates a closed set and a disjoint compact set: for each $y\in K\setminus V$, choose an open neighborhood $W_y$ with $\overline{W_y}\cap C=\emptyset$, and take a finite subcover $W_1,\ldots,W_m$ of $K\setminus V$. Then $U=X\setminus\bigcup_{j=1}^m\overline{W_j}$ is open, contains $C$, and satisfies $U\cap K\subseteq V$. If $K\setminus V$ is empty, take $U=X$ instead. Hence
\[ \{n\in \omega:x_n\in U\}\subseteq\{n\in \omega:x_n\notin K\}\cup\{n\in \omega:x_n\in V\}\in\mathcal I, \]
contradicting $F_\eta\subseteq U$ and
$\eta\in\Gamma_{\bm{x}}(\mathcal I,\mathscr F)$. 
\end{proof}

\medskip

\begin{proof}
 [Proof of Corollary \ref{cor:familiescluster}]  
 Since $X$ is a compact metric space, then it is regular and separable. In addition, there are no sequences $\bm{x}\in X^\omega$ for which $\Gamma_{\bm{x}}(\mathcal{I})=\emptyset$, see e.g. \cite[Proposition 5.1]{FKL2026}. Recalling that all sets of the type $\Gamma_{\bm{x}}(\mathcal{I})$ are closed, we obtain by the proof of (i) $\implies$ (ii) in \cite[Theorem 6.2]{FKL2026} (which works also in finite spaces) that 
 $$
 \mathscr{C}_X(\mathcal{I}, \mathscr{F}^\natural)=\left\{C\subseteq X: C \text{ is nonempty closed}\right\}.
 $$
 The conclusion follows applying Theorem \ref{thm:01thmclusterrough}.
\end{proof}

\medskip

\begin{proof}[Proof of Theorem \ref{thm:clustersetclosed}]
If $\Gamma_{\bm{x}}(\mathcal{I}, \mathscr{F})=\emptyset$, the claim is obvious. Otherwise, pick a $\tau$-convergent net $(\eta_i)_{i \in I}$ with values in $\Gamma_{\bm{x}}(\mathcal{I}, \mathscr{F})$ and define $\eta:=\lim_i \eta_i$. Since the map $\eta\mapsto F_\eta$ is $\widehat{\tau}$-continuous, then the net $(F_{\eta_i})_{i \in I}$ is $\widehat{\tau}$-convergent to $F_\eta$. Fix an arbitrary open set $U\subseteq X$ which contains $F_\eta$ and define the $\widehat{\tau}$-open set $\widehat{U}:=\{F \in \mathcal{H}(X): F\subseteq U\}$. By the convergence of $(F_{\eta_i})_{i \in I}$, there exists $j \in I$ such that $F_{\eta_j} \in \widehat{U}$, i.e., $F_{\eta_j}\subseteq U$. Since $\eta_j \in \Gamma_{\bm{x}}(\mathcal{I}, \mathscr{F})$, it follows that $\{n \in \omega: x_n\in U\} \in \mathcal{I}^+$. Therefore $\eta \in \Gamma_{\bm{x}}(\mathcal{I}, \mathscr{F})$.
\end{proof}

\medskip

\begin{proof}[Proof of Lemma \ref{lem:hattaucontinuous}] Fix $\eta\in X$ and an open set $U\supseteq B_{r(\eta)}(\eta)$. If $U=X$, there is nothing to prove. Otherwise, there exists $\delta>0$ such that $B_{r(\eta)+2\delta}(\eta)\subseteq U$. Indeed, if no such $\delta$ existed, one could choose $y_j\notin U$ with $d(y_j,\eta)\le r(\eta)+2^{-j}$ for all $j\in \omega$. Then, Heine--Borel property would give a subsequence converging to a point of $B_{r(\eta)}(\eta)\setminus U$, a contradiction. For all $\lambda$ sufficiently close to $\eta$, upper semicontinuity implies $r(\lambda)<r(\eta)+\delta$, and also $d(\lambda,\eta)<\delta$. Thus every $y\in B_{r(\lambda)}(\lambda)$ satisfies $d(y,\eta)<r(\eta)+2\delta$, so $B_{r(\lambda)}(\lambda)\subseteq U$. This is upper Vietoris continuity, concluding the proof. \end{proof}

\medskip 

\begin{proof}
[Proof of Proposition \ref{prop:basicIlimitpoints}]
    The claim is obvious if there are no $\eta \in X$ such that  $F_\eta \cap \Lambda_{\bm{x}}(\mathcal{I})\neq \emptyset$. Otherwise, pick $\eta,\eta^\prime \in X$ such that $\eta^\prime \in F_\eta \cap \Lambda_{\bm{x}}(\mathcal{I})$. Since $\eta^\prime$ is an $\mathcal{I}$-limit point of $\bm{x}$, there exists $S \in \mathcal{I}^+$ such that the subsequence $(x_n: n \in S)$ is convergent to $\eta^\prime$. It is enough to show that the same subsequence $(x_n: n \in S)$ witnesses that $\eta$ is an $(\mathcal{I}, \mathscr{F})$-limit point. To this aim, pick an open set $U$ containing $F_\eta$. In particular, $U$ is an open neighborhood of $\eta^\prime$. Hence $\{n \in S: x_n \notin U\}$ is finite. By the arbitrariness of $U$, this finishes the proof. 
\end{proof}

\medskip

\begin{proof}
[Proof of Theorem \ref{thm:ilimitpointcharact}]
Fix $\bm{x}\in X^\omega$ and define  $\bm{z}:=d(\bm{x},F_\eta)$. Define also $U_k:=\{y \in X: d(y,F_\eta)<2^{-k}\}$ for each $k \in \omega$. Then each $U_k$ is an open set containing $F_\eta$, and the sequence $(U_k: k \in \omega)$ is decreasing. Note also that $x_n \notin U_k$ if and only if $z_n \ge 2^{-k}$. 

\smallskip

\textsc{If part}. Suppose that $0\in \Lambda_{\bm{z}}(\mathcal{I})$.
Pick $S\in\mathcal{I}^+$ such that $(z_n:n\in S)$ converges to $0$. Then $(x_n:n\in S)$ is $\mathscr{F}$-convergent to $\eta$. Indeed, fix an open set $U\subseteq X$ containing $F_\eta$. Since both $C:=X\setminus U$ and $F_\eta$ are closed, by the \textsc{UC}-property of $X$ there exists $k \in \omega$ such that $F_\eta\subseteq U_k \subseteq U$. It follows that 
$$
\{n\in S:\ x_n\notin U\}\subseteq \{n\in S:\ x_n\notin U_k\}=\{n\in S: z_n\ge 2^{-k}\}\in\mathrm{Fin}.
$$
Therefore $\eta \in \Lambda_{\bm{x}}(\mathcal{I}, \mathscr{F})$. 

\medskip

\textsc{Only If part}. Suppose that $\eta\in \Lambda_{\bm{x}}(\mathcal{I},\mathscr{F})$. 
Pick $S\in\mathcal{I}^+$ such that $(x_n:n\in S)$ is $\mathscr{F}$-convergent to $\eta$. Then $(z_n:n\in S)$ converges to $0$. Indeed, $\{n \in S: z_n \ge 2^{-k}\}=\{n \in S: x_n \notin U_k\} \in \mathrm{Fin}$ for each $k \in \omega$. Therefore $0 \in \Lambda_{\bm{z}}(\mathcal{I})$. 
\end{proof}

\medskip

\begin{rmk}\label{rmk:ICproperytyequivalence}
    Set $\mathcal{I}:=\mathrm{Fin}$ and let $X$ be a metric space without the \textsf{UC}-property. Then there exist a point $\eta\in X$, a rough family $\mathscr{F}$ with $F_\eta$ closed, and a sequence $\bm{x} \in X^\omega$ such that $0 \in \Lambda_{\bm{z}}(\mathcal{I})$ and $\eta \notin \Lambda_{\bm{x}}(\mathcal{I}, \mathscr{F})$, where $\bm{z}:=d(\bm{x}, F_\eta)$. 

    To this end, since $X$ does not have the \textsf{UC}-property, it is possible to pick nonempty disjoint closed sets $A,B\subseteq X$ such that $d(A,B):=\inf\{d(x,y): x \in A, y \in B\}=0$. Fix $\eta \in A$ and define $F_\eta:=A$ and $F_{\lambda}:=\{\lambda\}$ for all $\lambda \in X\setminus \{\eta\}$. Pick also a sequence $\bm{x} \in B^\omega$ so that $x_n \in B$ and $d(x_n, A)<2^{-n}$ for each $n \in \omega$. Then $\mathscr F$ is a rough family, $F_\eta$ is closed, and $|z_n|<2^{-n}$ for all $n \in \omega$, so that $0 \in \Lambda_{\bm{z}}(\mathcal{I})$. On the other hand, if $S\subseteq \omega$ is infinite and we consider the open $U:=X\setminus B$ (which contains $A=F_\eta$), then $\{n \in S: x_n \notin U\}=S\in \mathrm{Fin}^+$. Hence $(x_n: n \in S)$ is not $\mathscr{F}$-convergent to $\eta$, i.e., $\eta \notin \Lambda_{\bm{x}}(\mathcal{I}, \mathscr{F})$. 
\end{rmk}

\medskip

\begin{proof}
[Proof of Proposition \ref{prop:existence}]
Consider the one-point compactification $X:=\omega \cup \{\omega\}$ endowed with the topology such that each $n \in \omega$ is isolated, and a base of open neighborhoods at $\omega$ is $\{\omega\} \cup (\omega \setminus [0,n])$ for some $n \in \omega$. In particular, since $X$ is homeomorphic to $\{0\}\cup \{2^{-n}: n \in\omega\}$, it is a compact metric space. Now, fix a bijection $\pi:\omega\to \omega\times\omega$.
For $A\subseteq \omega\times\omega$ and $k\in\omega$, set $A_k:=\{m\in\omega:(k,m)\in A\}$.
Define the Fubini ideal $\mathrm{Fin}\times\mathrm{Fin}$ on $\omega\times\omega$ by
$$
\mathrm{Fin}\times\mathrm{Fin}:=\Bigl\{A\subseteq \omega\times\omega: \bigl|\{k\in\omega: A_k\notin\mathrm{Fin}\}\bigr|<\infty\Bigr\}.
$$
Let $\mathcal I$ be its copy on $\omega$, namely, $
\mathcal I:=\{B\subseteq \omega: \pi[B]\in \mathrm{Fin}\times\mathrm{Fin}\}$. Define a rough family $\mathscr F=\{F_\eta:\eta\in X\}$ by 
$F_\omega:=\{\omega\}$ and $F_k:=\{k,\omega\}$ for all $k \in \omega$. 
Clearly, each $F_\eta$ is closed in $X$. Next, define the sequence $\bm{x} \in X^\omega$ as follows. For each $n\in\omega$ write
$\pi(n)=(k,m)$ and set
$$
x_n:=
\begin{cases}
\,k&\text{if $m$ is even},\\
\,m&\text{if $m$ is odd}.
\end{cases}
$$

\begin{claim}\label{claim1:Lim}
$\mathsf{Lim}_{\bm{x}}(\mathcal I^\star,\mathscr F)=\emptyset$.
\end{claim}
\begin{proof}
Towards a contradiction, suppose that there exists $\eta\in X$ with
$(\mathcal I^\star,\mathscr F)$-$\lim_n x_n=\eta$. Then there exists $S\in\mathcal I$ such that the subsequence
$(x_n:n\in \omega\setminus S)$ is $\mathscr F$-convergent to $\eta$.
Set $A:=\omega\setminus S$ and $S':=\pi[S]\in \mathrm{Fin}\times\mathrm{Fin}$.
By definition of $\mathrm{Fin}\times\mathrm{Fin}$, the set
$
E:=\{k\in\omega: S'_k\notin\mathrm{Fin}\}
$ 
is finite. Hence, for every $k\in\omega\setminus E$, the set $S'_k$ is finite and thus $A'_k:=\pi[A]_k=\omega\setminus S'_k$ is cofinite.
In particular, for each $k\in\omega\setminus E$, the set $A'_k$ contains infinitely many even integers, and therefore 
$\bigl|\{n\in A: x_n=k\}\bigr|=\infty$.

If $\eta=\omega$, fix $k\in\omega\setminus E$ and consider the open neighborhood
$U:=\{\omega\}\cup (\omega\setminus [0,k])$ of $\omega$; then $F_\omega\subseteq U$ and $k\notin U$.
Since $\{n\in A:x_n=k\}$ is infinite, we get $\{n\in A:x_n\notin U\}$ infinite, contradicting $\mathscr F$-convergence to $\omega$.

If $\eta=k_0\in\omega$, pick $k\in\omega\setminus(E\cup\{k_0\})$ and let 
$U:=\{k_0,\omega\}\cup (\omega\setminus [0,k])$.  
Then $U$ is an open set containing $F_{k_0}=\{k_0,\omega\}$, while $k\notin U$.
As in the previous case, $\{n\in A:x_n=k\}$ is infinite, so $\{n\in A:x_n\notin U\}$ is infinite, contradicting $\mathscr F$-convergence to $k_0$.
\end{proof}

\begin{claim}\label{claim2:Lim}
$\mathsf{Lim}_{\bm{x}}(\mathcal I,\mathscr F)=X$.
\end{claim}
\begin{proof}
Fix $\eta\in X$ and let $U\subseteq X$ be an open set containing $F_\eta$. We need to show that $\{n \in \omega: x_n \notin U\} \in \mathcal{I}$. 

First, suppose that $\eta=\omega$. Suppose also without loss of generality that $U$ is a basic open neighborhood of $\omega$, so that there exists $t \in \omega$ such that $U=\{\omega\}\cup (\omega\setminus [0,t))$. For each $j \in \omega \cap [0,t)$, define $B_j:=\{n \in \omega: x_n=j\}$. Observe that $\pi[B_j]$ has infinitely many points only in column $j$ (since $(j,2i)\in\pi[B_j]$ for all $i$),
while for $k\neq j$ we have $|\pi[B_j]_k|\le 1$ (the only possible point is $(k,j)$, and this occurs only if $j$ is odd).
Hence $\pi[B_j]\in \mathrm{Fin}\times\mathrm{Fin}$ and so $B_j\in\mathcal I$. It follows that 
$
\{n \in \omega: x_n \notin U\}=\bigcup_{j<t}B_j \in \mathcal{I}.
$ 

Next, suppose that $\eta=k_0\in\omega$, then $U$ is a neighborhood of $\omega$ and contains $k_0$; hence there exists $t\in\omega$ with
$\{\omega\} \cup (\omega\setminus [0,t)) \subseteq U$ and $k_0\in U$. With the same notation as in the previous case, we conclude that $\{n \in \omega: x_n \notin U\}\subseteq \bigcup_{j<t, j\neq k_0}B_j \in \mathcal{I}$. 
\end{proof}

\begin{claim}\label{claim3:Lambda}
$\Lambda_{\bm{x}}(\mathcal I,\mathscr F)=X$.
\end{claim}
\begin{proof}
Define the set $T:=\pi^{-1}[T']\subseteq \omega$, where 
\[
T':=\{(k,m)\in\omega\times\omega: m \text{ is odd and } m\ge k\}.
\]
Since $T'_k$ is infinite for every $k \in \omega$, it follows that 
$T'\notin \mathrm{Fin}\times\mathrm{Fin}$ and thus $T\in\mathcal I^+$. 
We claim that the subsequence $(x_n:n\in T)$ converges (in the usual sense) to $\omega$. 
Indeed, fix $q\in\omega$. If $n\in T$ and $x_n<q$, then $\pi(n)=(k,m)\in T'$ and $x_n=m$,
so $m<q$ is odd and $k\le m<q$. Hence there are only finitely many possibilities for $(k,m)$, and therefore
$\{n\in T:x_n<q\}$ is finite. 

Now fix $\eta\in X$ and let $U\subseteq X$ be an open set containing $F_\eta$.
Since $\omega\in F_\eta$, $U$ is a neighborhood of $\omega$, so $\{\omega\}\cup (\omega\setminus [0,q))\subseteq U$ for some $q \in \omega$.
By the previous observation, we obtain that $\{n\in T: x_n\notin U\}\subseteq \{n\in T:x_n<q\}\in \mathrm{Fin}$.
Therefore $(x_n:n\in T)$ is $\mathscr F$-convergent to $\eta$. Since $T\in\mathcal I^+$, this proves that $\eta\in \Lambda_{\bm{x}}(\mathcal I,\mathscr F)$.
\end{proof}

The conclusion follows putting together Claims \ref{claim1:Lim}--\ref{claim3:Lambda}, together with the trivial inclusion $\mathsf{Lim}_{\bm{x}}(\mathcal{I}, \mathscr{F})
\subseteq \Gamma_{\bm{x}}(\mathcal{I}, \mathscr{F})$, so that also $\Gamma_{\bm{x}}(\mathcal{I}, \mathscr{F})=X$.
\end{proof}

\begin{rmk}\label{rmk:hattaucontinuous}
The same construction above satisfies also the following two additional properties: (i) there exist a point $\eta \in X$ and a sequence $\bm{y} \in X^\omega$ such that $\lim_k y_k=\eta$ and $y_k \notin F_\eta$ for all $k \in \omega$; and (ii) the map $\eta \mapsto F_\eta$ (denoted by $h$) 
is $\widehat{\tau}$-continuous. 

To show (i), it is enough to consider the sequence $\bm{y}$ defined by $y_k:=k$ for all $k \in \omega$ and $\eta=\omega$. To show (ii), pick an open set $U\subseteq X$ and consider the basic $\widehat{\tau}$-open set $\widehat{U}:=\{F \in \mathcal{H}(X): F\subseteq U\}$. We need to show that 
$S:=h^{-1}[\widehat{U}]=\{\eta \in X: F_\eta \subseteq U\}$ is open in $X$. If $F_\eta \subseteq U$ and $\eta =k \in \omega$ then $\eta$ is isolated, hence it is in the interior of $S$. If $F_\eta \subseteq U$ and $\eta=\omega$ then $\omega \in U$, hence there exists $n\in\omega$ such that $V:=\{\omega\}\cup(\omega\setminus[0,n])\subseteq U$. To conclude, it is enough to observe that $F_\lambda \subseteq U$ for all $\lambda \in V$: indeed, if $\lambda=\omega$ this is clear, while if $\lambda \in \omega\setminus [0,n]$ then $F_\lambda=\{\lambda,\omega\}\subseteq U$. Hence $V\subseteq S$. 
\end{rmk}

\medskip 

\begin{proof}
    [Proof of Theorem \ref{thm:closure}]
    First, let us show that $c^b(\mathcal{I}, \mathscr{F})$ is closed in $\ell_\infty$. To this aim, pick a sequence $(\bm{x}^{(m)}: m\in\omega)\in c^b(\mathcal{I}, \mathscr{F})^\omega$ converging to some $\bm{x} \in\ell_\infty$. For each $m\in\omega$, fix a real $\eta_m\in\mathsf{Lim}_{\bm{x}^{(m)}}(\mathcal I,\mathscr{F})$. 

    Hereafter, we will use the notation $d(A,B):=\inf\{|x-y|: x\in A, y \in B\}$ for all nonempty $A,B\subseteq \mathbb{R}$, and $d(y,B):=d(\{y\}, B)$. Since the sets $F_\eta$ are uniformly bounded by hypothesis, 
    it is possible to fix $r \in [0,\infty)$ such that $F_\eta\subseteq B_r(\eta)$ for all $\eta \in \mathbb{R}$. 

    \begin{claim}\label{claim:etambounded}
    The sequence $(\eta_m: m \in \omega)$ is bounded. 
    \end{claim}
    \begin{proof}
    Fix $m\in \omega$ and define $U_m:=\{y \in \mathbb{R}: d(y,F_{\eta_{m}})<1\}$. Then $U_m$ is an open set containing $F_{\eta_m}$, hence $D_m:=\{n \in \omega: x^{(m)}_n \notin U_m\} \in \mathcal{I}$. In particular, it is possible to fix an integer $n_m \in \omega\setminus D_m$. Thus $d(x^{(m)}_{n_m}, F_{\eta_m})<1$.  It follows that 
    $$
    |\eta_m| \le |x^{(m)}_{n_m}|+|x^{(m)}_{n_m}-\eta_m| \le \|\bm{x}^{(m)}\|_\infty+r+1.
    $$
    The claim follows since $(\bm{x}^{(k)}: k\in\omega)$ converges, so that $\sup_k \|\bm{x}^{(k)}\|_\infty<\infty$. 
    \end{proof}

Thanks to Claim \ref{claim:etambounded} and the Heine--Borel property of $\mathbb{R}$, there exists a subsequence $(\eta_{m_j}: {j\in\omega})$
converging to some $\eta\in X$. 

\begin{claim}\label{claim:etalimitpoint}
    $\eta\in\mathsf{Lim}_{\bm{x}}(\mathcal I,\mathscr{F})$. 
\end{claim}
\begin{proof}
Fix an open set $U\subseteq\mathbb R$ containing $F_\eta$.
If $U=\mathbb R$, the required exceptional set is empty, so suppose
$U\ne\mathbb R$. The set $F_\eta$ is closed and contained in $B_r(\eta)$, hence compact. Since $\mathbb R\setminus U$ is nonempty, closed, and disjoint from $F_\eta$, we may set
\[
\varepsilon=d(F_\eta,\mathbb R\setminus U)>0.
\]
Define the open set 
$
V:=
\{y\in X: d(y,F_\eta)<\varepsilon/2\},
$ 
so that $F_\eta\subseteq V\subseteq \overline V\subseteq U$. Since $\lim_j \eta_{m_j}=\eta$ and the map $\lambda\mapsto F_\lambda$ is $\widehat{\tau}$-continuous, it follows that there exists $j_0 \in \omega$ such that $F_{\eta_{m_j}}\subseteq V$ for all $j\ge j_0$. Without loss of generality, we can also suppose that $\|\bm{x}^{(m_j)}-\bm{x}\|_\infty<\varepsilon/2$ for all $j\ge j_0$.

Now, suppose that $n \in \omega$ is an integer such that $x_n \notin U$, so that $d(x_n,F_\eta) \ge d(X\setminus U, F_\eta) \ge \varepsilon$. Then necessarily $x_n^{(m_{j_0})} \notin V$: indeed, in the opposite, we would have 
$$
d(x_n,F_\eta) \le |x_n-x_n^{(m_{j_0})}|+d\left(x_n^{(m_{j_0})},F_\eta\right)\le \|\bm{x}^{(m_{j_0})}-\bm{x}\|_\infty+\frac{\varepsilon}{2}<\varepsilon,
$$
which is the required contradiction. 
Since $\eta_{m_{j_0}} \in \mathsf{Lim}_{\bm{x}^{({m_{j_0}})}}(\mathcal{I}, \mathscr{F})$, this implies that 
$$
\{n \in \omega: x_n \notin U\} \subseteq \{n \in \omega: x_n^{(m_{j_0})} \notin V\}\in \mathcal{I}.
$$
By the arbitrariness of $U$, we conclude that $\eta=(\mathcal{I}, \mathscr{F})\text{-}\lim_n x_n$. 
\end{proof}

Claim \ref{claim:etalimitpoint} proves that the sequence $\bm{x}$ (which is the limit of $(\bm{x}^{(k)}: k\in\omega)$ in $\ell_\infty$) satisfies $\mathsf{Lim}_{\bm{x}}(\mathcal I,\mathscr{F})\neq \emptyset$. Therefore $c^b(\mathcal{I}, \mathscr{F})$ is closed in $\ell_\infty$. 

\medskip

At this point, since $c^b(\mathcal{I}, \mathscr{F})$ is a closed set containing $c^b(\mathcal{I}^\star, \mathscr{F})$, then it contains also its closure. To complete the proof, it is sufficient to show the converse inclusion (which indeed holds for every rough family $\mathscr{F}$). 

Fix $\bm y\in c^b(\mathcal I,\mathscr F)$, choose $\eta\in\mathsf{Lim}_{\bm y}(\mathcal I,\mathscr F)$, and fix $k\in\omega$. Set $\delta=2^{-k-1}$ and $U:=\{z\in\mathbb R:d(z,F_\eta)<\delta\}$. Then $A=\{n:y_n\notin U\}\in\mathcal I$. Put $x_n=y_n$ for $n\in A$, and choose $x_n\in F_\eta$ with $|x_n-y_n|<\delta$ for $n\notin A$. The resulting sequence is bounded and satisfies $\|\bm x-\bm y\|_\infty\le\delta<2^{-k}$. On $\omega\setminus A$ all its terms belong to $F_\eta$, so $\eta\in\mathsf{Lim}_{\bm x}(\mathcal I^\star,\mathscr F)$. This proves the required density.
\end{proof}

\medskip

\begin{proof}
    [Proof of Theorem \ref{tmm:uppersemicrimpliesGammaLimcont}]
    First, we show that $\Phi$ is $\widehat{\tau}$-continuous. 
    Fix a bounded real sequence $\bm{x} \in c^b(\mathcal{I}, \mathscr{F})$ and an open set $U\subseteq \mathbb{R}$ such that $\Phi(\bm{x})\subseteq U$. 
    We need to show that there exists $\varepsilon>0$ such that $\Phi(\bm{y})\subseteq U$ for all $\bm{y} \in c^b(\mathcal{I}, \mathscr{F})$ with $\|\bm{x}-\bm{y}\|_\infty<\varepsilon$. This is obvious if  $U=\mathbb{R}$, hence we suppose hereafter that $U^c:=\mathbb{R}\setminus U$ is nonempty. 

    Define the map $m_{\bm{x}}: \mathbb{R}\to \mathbb{R}$ by 
    \begin{equation}\label{eq:definitionmx}
    \forall \eta \in \mathbb{R}, \quad 
    m_{\bm{x}}(\eta):=\max\{|a-\eta|: a \in \Gamma_{\bm{x}}(\mathcal{I})\}-r(\eta)
    \end{equation}
    Observe that $\Gamma_{\bm{x}}(\mathcal{I})$ is nonempty compact, and the map $\eta \mapsto |a-\eta|$ is continuous for each $a$, hence $m_{\bm{x}}$ is well defined. In addition, the map $\eta \mapsto \max\{|a-\eta|: a \in \Gamma_{\bm{x}}(\mathcal{I})\}$ is continuous, see e.g. Berge's maximum theorem \cite[Theorem 17.31]{MR2378491}. Hence $m_{\bm{x}}$ is lower semicontinuous. It follows by \cite[Corollary 1.4]{LeonettiJCA} that 
    $$
    \Phi(\bm{x})=\{\eta \in \mathbb{R}: m_{\bm{x}}(\eta)\le 0\}.
    $$ 
    In particular, $m_{\bm{x}}(\eta)>0$ for all $\eta \in U^c$. 
    Now, define the compact set $K:=U^c \cap [-R,R]$, where $R:=1+\|\bm{x}\|_\infty+\sup_{\eta}r(\eta)$. Taking into account that a lower semicontinuous function attains its minimum on compact sets, it is possible to define the positive real $\alpha:=\min\{m_{\bm{x}}(\eta): \eta \in K\}>0$, provided that $K\neq \emptyset$. At this point, if $K=\emptyset$ it is enough to choose $\varepsilon:=\nicefrac{1}{2}$; while if $K\neq \emptyset$, we choose
    $$
    \varepsilon \in \left(0, \min\left\{1,\, \frac{\alpha}{2}\right\}\right).
    $$

    In fact, pick a bounded real sequence $\bm{y} \in c^b(\mathcal{I}, \mathscr{F})$ such that $\|\bm{x}-\bm{y}\|_\infty<\varepsilon$, and define the map $m_{\bm{y}}$ analogously as in \eqref{eq:definitionmx}. Now, fix $\eta \in U^c$. We claim 
    that 
    $m_{\bm{y}}(\eta)>0$. This follows considering the two cases below (case (ii) is considered only if $K\neq \emptyset$): 
    \begin{enumerate}[label={\rm (\roman{*})}]
        \item If $|\eta|>R$ then 
        \begin{displaymath}
            \begin{split}
                m_{\bm{y}}(\eta)
                &=\max\{|a-\eta|: a \in \Gamma_{\bm{y}}(\mathcal{I})\}-r(\eta)\\
                &\ge |\eta|-\|\bm{y}\|_\infty -r(\eta) 
                > R - (\|\bm{x}\|_\infty +1) - \sup_{z \in \mathbb{R}}r(z)=0.
            \end{split}
        \end{displaymath}

        \item Suppose now that $|\eta| \le R$ and $\eta \in U^c$ (that is, $\eta \in K$). Recall that the Hausdorff metric $d_{\mathsf{H}}$ on the family $\mathcal{K}(\mathbb{R})$ of nonempty compact subsets of $\mathbb{R}$ is defined by $d_{\mathsf{H}}(A,B):=\max\{\sup_{a \in A}d(a,B), \sup_{b \in B}d(A,b)\}$ for all $A,B \in \mathcal{K}(\mathbb{R})$. Observe also the elementary inequality 
        $$
        \forall A,B \in \mathcal{K}(\mathbb{R}), \forall z \in \mathbb{R}, \quad 
        \left|\sup_{b \in B}|b-z|-\sup_{a \in A}|a-z|\right|\le d_{\mathsf{H}}(A,B).
        $$
        Indeed, $|\sup_B|b-z|-\sup_A|a-z||\le d_{\mathsf H}(A,B)$
follows by approximating each point of either compact set by a point of
the other. We also have
\[
d_{\mathsf H}(\Gamma_{\bm x}(\mathcal I),\Gamma_{\bm y}(\mathcal I))
\le\|\bm x-\bm y\|_\infty.
\]
To see this, put $\delta=\|\bm x-\bm y\|_\infty$ and fix
$a\in\Gamma_{\bm x}(\mathcal I)$ and $t>0$.
The $\mathcal I$-positive set $\{n:|x_n-a|<t\}$ is contained in
$\{n:y_n\in[a-\delta-t,a+\delta+t]\}$.
Compactness of this interval and the compact-visit property of ideal
cluster points imply that it meets $\Gamma_{\bm y}(\mathcal I)$.
Letting $t\downarrow0$, it follows that 
$d(a,\Gamma_{\bm y}(\mathcal I))\le\delta$; the converse  direction is
symmetric. Hence 
\[
m_{\bm y}(\eta)\ge m_{\bm x}(\eta)-\delta
\ge\alpha-\delta>0.
\]
%
    \end{enumerate}

It follows that $\Phi(\bm{y})=\{\eta \in \mathbb{R}: m_{\bm{y}}(\eta)\le 0\}\subseteq U$. Therefore $\Phi$ is $\widehat{\tau}$-continuous. 

\medskip

The proof that also the map $\Psi$ is $\widehat{\tau}$-continuous proceeds along the same lines, once we replace $c^b(\mathcal{I}, \mathscr{F})$ with $\ell_\infty$, the function $m_{\bm{x}}$ in \eqref{eq:definitionmx} with 
$$
\forall \eta \in \mathbb{R}, \quad 
\tilde{m}_{\bm{x}}(\eta):=\min\{|a-\eta|: a \in \Gamma_{\bm{x}}(\mathcal{I})\}-r(\eta),
$$
and we note that $\Psi(\bm{x})=\{\eta \in \mathbb{R}: \tilde{m}_{\bm{x}}(\eta)\le 0\}$ thanks to Theorem \ref{thm:01thmclusterrough} (and similarly for the sequence $\bm{y}$); 
we omit further details. 
\end{proof}

\medskip

\begin{proof}
[Proof of Theorem \ref{thm:characterizationPplus}]  
\textsc{If part.}
Suppose that $\mathcal{I}$ is not a $P^+$-ideal. 
We need to show that there exists a sequence $\bm{x} \in X^\omega$ such that $\Lambda_{\bm{x}}(\mathcal{I},\mathscr{F})\neq \Gamma_{\bm{x}}(\mathcal{I},\mathscr{F})$.

Since $\mathcal{I}$ is not $P^+$, there exists a decreasing sequence $(A_k:k\in\omega)$ in $\mathcal{I}^+$ such that no $A\in\mathcal{I}^+$ satisfies $A\setminus A_k\in\mathrm{Fin}$ for all $k$. We can assume without loss of generality that $A_0=\omega$ and the sequence $\bm{y}$ is injective. 
Define $\bm{x} \in X^\omega$ by $x_n:=y_k$ for all $n,k \in \omega$ with $n \in A_k\setminus A_{k+1}$, and $x_n=\eta$ for all $n \in 
\bigcap_t A_t$. 
Then we claim that
$$
\eta \in \Gamma_{\bm{x}}(\mathcal{I},\mathscr{F})\setminus \Lambda_{\bm{x}}(\mathcal{I},\mathscr{F}).
$$
In fact, if $U$ is an open set containing $F_\eta$, then there exists $k_0 \in \omega$ with $y_{k} \in U$ for all $k\ge k_0$, so that $A_{k_0}\subseteq \{n \in \omega: x_n \in U\}$. Hence the latter set is $\mathcal{I}$-positive. Thus $\eta \in \Gamma_{\bm{x}}(\mathcal{I},\mathscr{F})$.  
At the same time, pick an infinite set $A\subseteq \omega$ such that $(x_n: n \in A)$ is $\mathscr{F}$-convergent to $\eta$. 
Observe that 
the set 
$
U_k:=X\setminus \{y_0,\dots,y_{k-1}\}
$
is open for each $k \in \omega$. Moreover, by the standing hypothesis, $F_{\eta}\subseteq U_k$, so that $\{n \in A: x_n\notin U_k\} \in \mathrm{Fin}$ for each $k \in \omega$. Since $A_0=\omega$ and by construction $A_k=\{n \in \omega: x_n \in U_k\}$, it follows that 
$
A\setminus A_k=A\setminus \{n \in \omega: x_n \in U_k\} \in \mathrm{Fin}
$ 
for all $k \in \omega$. 
Hence $A \in \mathcal{I}$, so that $\eta \notin \Lambda_{\bm{x}}(\mathcal{I},\mathscr{F})$. 

\medskip

\textsc{Only If part.}
Suppose that $\mathcal{I}$ is a $P^+$-ideal and pick a sequence $\bm{x}\in X^\omega$. 
We need to show that $\Gamma_{\bm{x}}(\mathcal{I},\mathscr{F})\subseteq \Lambda_{\bm{x}}(\mathcal{I},\mathscr{F})$ (since the converse is obvious). 

The claim is clear if $\Gamma_{\bm{x}}(\mathcal{I},\mathscr{F})=\emptyset$. Otherwise, fix $\lambda \in \Gamma_{\bm{x}}(\mathcal{I},\mathscr{F})$ and define the sets $U_k:=\{y \in X: d(y,F_\lambda)<2^{-k}\}$ and $A_k:=\{n \in \omega: x_n\in U_k\}$ for all $k \in \omega$. Then each $U_k$ is an open set containing $F_\lambda$ and, since $\lambda$ is an $(\mathcal{I},\mathscr{F})$-cluster point, $(A_k: k \in \omega)$ is a decreasing sequence in $\mathcal{I}^+$. 
By the $P^+$-property, there exists $A\in\mathcal{I}^+$ such that $A\setminus A_k$ is finite for all $k$. It is enough to show that the subsequence $(x_n: n\in A)$ is $\mathscr{F}$-convergent to $\lambda$. Indeed, suppose that $U\subseteq X$ an open set containing $F_\lambda$. Since $X$ is a compact metric space, it has the \textsc{UC}-property, hence the disjoint closed sets $F_\lambda$ and $X\setminus U$ are at a positive distance apart. So there exists $k \in \omega$ such that $F_\lambda \subseteq U_k\subseteq U$. It follows that 
\[
\{n\in A: x_n\notin U\}\subseteq \{n\in A: x_n\notin U_k\}=A\setminus A_k\in\mathrm{Fin}.
\]
Therefore $\lambda \in \Lambda_{\bm{x}}(\mathcal{I},\mathscr{F})$. 
\end{proof}

\subsection{Note added in proof}

In the list below, we record several corrections to \cite{LeonettiJCA}. 
Here, $\mathsf{L}_{\bm{x}}(\mathcal I,\mathscr F)$ denotes the set written as $\mathsf{Lim}_{\bm{x}}(\mathcal I,\mathscr F)$ in the present work, and we recall
$$\mathrm{core}_{\bm{x}}(\mathcal{I}):=\bigcap_{E \in \mathcal{I}}\overline{\mathrm{co}}(\{x_n: n\in \omega\setminus E\}), 
$$
where $\overline{\mathrm{co}}(\cdot)$ denotes the closed convex hull operator, see \cite{MR3955010}. 
The author thanks Alifa Othman Aljamili and Mehmet K\"u\c{c}\"ukaslan (Mersin University, Turkey) 
for their careful reading and helpful comments concerning the earlier article and, indirectly, an improvement of this work. 
The author is also grateful to an anonymous referee for the observations on this manuscript. 

\begin{enumerate}[label={\rm (\alph{*})}]
\item \label{item:1corr} \cite[Theorem 1.8]{LeonettiJCA}: In its statement, the
``\textsc{UC}-property'' should be replaced by the ``Heine--Borel
property'', or equivalently, $X$ should be assumed to be
finite-dimensional. Indeed, no nonzero normed vector space
has the \textsc{UC}-property. To see this, fix a unit vector
$v\in X$ and consider the nonempty disjoint closed sets
$
A:=\{a_nv:n\in \omega\}$ 
and 
$B:=\{b_nv:n\in \omega\}$, 
where $a_n:=n$ and $b_n:=n+2^{-n-1}$ for all $n \in \omega$. 
Since $\lim_n (a_n-b_n)v=0$, 
these sets have distance zero. Thus the original hypotheses
allow only $X=\{0\}$.
In addition, the minus sign in the first displayed equality
in the original proof should be a plus sign.

More explicitly, let $\bm{x}$ be a sequence in
a finite-dimensional normed vector space $X$, let $\mathcal I$
be an ideal on $\omega$, and fix a concave function
$r:X\to[0,\infty)$. Define $\mathscr F$ by
$F_\eta:=B_{r(\eta)}(\eta)$ for every $\eta\in X$.
Then $\mathsf L_{\bm{x}}(\mathcal I,\mathscr F)$ is convex.

Indeed, the claim is immediate if this set is empty.
Otherwise, fix
$\eta_0\in\mathsf L_{\bm{x}}(\mathcal I,\mathscr F)$ and put 
$
K:=B_{r(\eta_0)+1}(\eta_0).
$ 
Finite dimensionality ensures that $K$ is compact.
Moreover, the open ball centered at $\eta_0$ with radius
$r(\eta_0)+1$ contains $F_{\eta_0}$, hence we get 
$\{n:x_n\notin K\}\in\mathcal I$.
Thus, we obtain by
\cite[Corollary 1.4]{LeonettiJCA} that  
$
\mathsf L_{\bm{x}}(\mathcal I,\mathscr F)
=\{\eta\in X:\Gamma_{\bm{x}}(\mathcal I)\subseteq B_{r(\eta)}(\eta)\}.
$ 
Now, pick
$\eta,\eta'\in\mathsf L_{\bm{x}}(\mathcal I,\mathscr F)$,
fix $\alpha\in(0,1)$, and define
$\gamma:=\alpha\eta+(1-\alpha)\eta'$.
For every $a\in \Gamma_{\bm{x}}(\mathcal I)$, it follows that 
\[
\begin{split}
\|a-\gamma\|
\le\alpha\|a-\eta\|+(1-\alpha)\|a-\eta'\|\le\alpha r(\eta)+(1-\alpha)r(\eta')
\le r(\gamma).
\end{split}
\]
Therefore $\Gamma_{\bm{x}}(\mathcal I)\subseteq B_{r(\gamma)}(\gamma)$, which shows that 
$\gamma\in\mathsf L_{\bm{x}}(\mathcal I,\mathscr F)$.

The same argument proves the convexity of $\mathsf L_{\bm{x}}(\mathcal I,\mathscr F)$ in any
normed vector space if one assumes directly that
$\{n:x_n\notin K\}\in\mathcal I$ for some compact $K\subseteq X$.
\item \label{item:2corr} \cite[Corollary 1.10]{LeonettiJCA}: In its statement, ``metric vector space'' must be replaced by
``normed vector space'', and
\(\mathsf L_{\bm{x}}(\mathrm{Fin},\mathscr F)\) in (iii) was meant to be \(\mathsf L_{\bm{x}}(\mathcal I,\mathscr F)\). 

More explicitly, let $\bm{x}$ be a sequence taking values in a normed vector space $X$ such that $\{n \in \omega: x_n \notin K\} \in \mathcal{I}$ for some compact $K\subseteq X$. Also, let $\mathcal{I}$ be an ideal on $\omega$, fix $r \in [0,\infty)$, and define $\mathscr{F}$ by $F_\eta=B_r(\eta)$ for all $\eta \in X$. Then the following are equivalent\textup{:}
\begin{enumerate}[label={\rm (\roman{*})}]
\item \label{eq:condition1} $\bm{x}$ is $\mathcal{I}$-convergent to some $\eta \in X$\textup{;}
\item \label{eq:condition2} $\mathrm{core}_{\bm{x}}(\mathcal{I})$ contains a unique vector $\eta^\prime \in X$\textup{;}
\item \label{eq:condition3} $\mathsf{L}_{\bm{x}}(\mathcal{I}, \mathscr{F})=B_r(\eta^{\prime\prime})$ for some $\eta^{\prime\prime} \in X$.
\end{enumerate}
In addition, in such case, $\eta=\eta^\prime=\eta^{\prime\prime}$. 

The equivalence \ref{eq:condition1} $\Longleftrightarrow$ \ref{eq:condition2} follows by \cite[Proposition 4.2]{MR3955010} (this holds, more generally, in locally convex metrizable topological vector spaces). 
Now, observe that $\Gamma_{\bm{x}}(\mathcal I)$ is a nonempty compact subset of $C:=\mathrm{core}_{\bm{x}}(\mathcal I)$, and by \cite[Theorem 1.9]{LeonettiJCA} we have $\mathsf L_{\bm{x}}(\mathcal I,\mathscr F)
=\{\gamma\in X:C\subseteq B_r(\gamma)\}$. 
If \ref{eq:condition2} holds, then $C=\{\eta'\}$, and therefore 
$ 
\mathsf L_{\bm{x}}(\mathcal I,\mathscr F)
=B_r(\eta'),
$ 
which proves \ref{eq:condition3}. 
Conversely, suppose that \ref{eq:condition3} holds. If there exists 
$a\in C\setminus\{\eta''\}$, define
\[
\gamma:=\eta''-r\frac{a-\eta''}{\|a-\eta''\|}.
\]
Since $\|\gamma-\eta''\|=r$, we have
$\gamma\in B_r(\eta'')
=\mathsf L_{\bm{x}}(\mathcal I,\mathscr F)$.
This implies that
$C\subseteq B_r(\gamma)$, hence $\|a-\gamma\|\le r$.
However,
\[
\|a-\gamma\|
=\left\|a-\eta''+
r\frac{a-\eta''}{\|a-\eta''\|}\right\|
=\|a-\eta''\|+r>r,
\]
which is a contradiction. Thus $C\subseteq\{\eta''\}$.
Since $C$ is nonempty, $C=\{\eta''\}$, proving \ref{eq:condition2} and
$\eta'=\eta''$. Therefore also \ref{eq:condition2} $\Longleftrightarrow$ \ref{eq:condition3}, and $\eta=\eta'=\eta''$. 

In fact, Example \ref{example:mehmet} below shows that the equivalence \ref{eq:condition2} $\Longleftrightarrow$ \ref{eq:condition3} fails in locally convex metrizable topological vector spaces. 

\item \cite[Corollary 1.7]{LeonettiJCA} has been already fixed in Corollary \ref{cor:UCpropertycluster}.
\end{enumerate}

\begin{example}\label{example:mehmet}
Let $X=\mathbb R$, equipped with the metric
\[
\forall s,t \in \mathbb{R}, \qquad 
d(s,t):=
\begin{cases}
\,|s-t| \,\,\,\,\,& \text{ if }st\ge0,\\
\,\max\{|s|,|t|\} & \text{ otherwise}.
\end{cases}
\]
Indeed, this is the metric induced by the injective map
$t\mapsto(\max\{t,0\},\max\{-t,0\})$ into $\mathbb R^2$
with its supremum norm. Moreover, $|s-t|/2\le d(s,t)\le |s-t|$ for all $s,t \in \mathbb{R}$, hence $d$ induces the usual Euclidean topology. Thus $X$ is a 
locally convex metrizable topological vector space.

Now, set $\mathcal I=\mathrm{Fin}$, $r=1$, and consider the rough family $\mathscr{F}$ defined by 
$F_\eta:=B_1(\eta)$ for every $\eta\in X$.
Consider the bounded real sequence
\[
\bm{x}=(0,1,0,1,0,\ldots).
\]
Then $\mathrm{core}_{\bm{x}}(\mathrm{Fin})=[0,1]$. 
On the other hand, $\bm{x}$ is $(\mathrm{Fin}, \mathscr{F})$-convergent to $\eta$ exactly when $\{0,1\}\subseteq F_\eta$. 
Indeed, membership of both values in $F_\eta$ suffices; if either value is absent, its complement is an open superset of $F_\eta$ missed by infinitely many terms. Now $d(0,\eta)=|\eta|$, and $d(1,\eta)\le1$ whenever $|\eta|\le1$. 
Hence
\[
\mathsf L_{\bm{x}}(\mathrm{Fin},\mathscr F)
=\{\eta\in\mathbb R:d(0,\eta)\le1,\ d(1,\eta)\le1\}
=[-1,1]=B_1(0).
\]
Therefore condition \ref{eq:condition3} holds, whereas condition \ref{eq:condition2} fails.
\end{example}

Lastly, we clarify two arguments from the earlier work:
\begin{enumerate}[label={\rm (\alph{*})}]
\setcounter{enumi}{3}

\item \label{item:3corr}
\emph{Proof of} \cite[Proposition 1.2]{LeonettiJCA}: 
The deduction that every nonempty $\nu$-open set equals
$\mathbb R$ requires propagation in both directions.
The following shorter argument proves the statement without
distinguishing maximal and nonmaximal ideals.

Define $F_\eta:=[\eta-1,\eta+1]$ for every $\eta\in\mathbb R$,
and fix an ideal $\mathcal I$ on $\omega$.
For each $a\in\mathbb R$, let $\bm{c}_a$ denote the constant
sequence with value $a$. Then 
$
\mathsf L_{\bm{c}_a}(\mathcal I,\mathscr F)=[a-1,a+1].
$ 
Suppose that a proper ideal $\mathcal J$ and a topology $\nu$
satisfy the claimed equivalence. Then $\bm{c}_a$ is 
$(\mathcal J,\nu)$-convergent to $b$ exactly when every
$\nu$-open neighborhood of $b$ contains $a$, since its
exceptional index set is either empty or all of $\omega$.
Now, $\bm{c}_0$ is $(\mathcal{I}, \mathscr{F}, \tau)$-convergent to $1$, and $\bm{c}_1$
is $(\mathcal{I}, \mathscr{F}, \tau)$-convergent to $2$. Thus every $\nu$-open neighborhood
of $2$ contains $1$, and every $\nu$-open neighborhood of $1$
contains $0$. Hence, every $\nu$-open neighborhood
of $2$ contains $0$, so $\bm{c}_0$ is
$(\mathcal J,\nu)$-convergent to $2$. This contradicts
$2\notin\mathsf L_{\bm{c}_0}(\mathcal I,\mathscr F)=[-1,1]$.

\item \label{item:4corr}
\emph{Proof of} \cite[Proposition 1.5]{LeonettiJCA}: 
The construction ensures that
$\eta\in\mathsf L_{\bm{x}}(\mathcal I,\mathscr F)$, and not
necessarily also $\eta'$ belongs to this set.
The necessity argument can be completed as follows. 

Since $\mathcal I$ is nonmaximal, choose a partition
$\{A,B\}$ of $\omega$ with $A,B\notin\mathcal I$.
Suppose that $\eta'\in F_\eta\setminus\{\eta\}$ for some
$\eta\in X$, and define $x_n:=\eta$ for $n\in A$ and
$x_n:=\eta'$ for $n\in B$. Then $\bm{x}$ is $(\mathcal{I}, \mathscr{F})$-convergent to $\eta$. 
Define
\[
D:=\sup_{\lambda\in X}\operatorname{diam}(F_\lambda)<\infty
\quad \text{ and }\quad 
k:=1+\left\lfloor\frac{D}{\|\eta-\eta'\|}\right\rfloor.
\]
If the family of $(\mathcal{I}, \mathscr{F})$-convergent sequences were a vector
space, then $k\bm{x}$ would be $(\mathcal{I}, \mathscr{F})$-convergent  to some
$\lambda\in X$. 
Note that $k\eta, k\eta'\in F_\lambda$: 
if a value $v\in\{k\eta,k\eta'\}$ were absent from
$F_\lambda$, then $X\setminus\{v\}$ would be an open set
containing $F_\lambda$ whose exceptional index set contains
$A$ or $B$, contradicting $(\mathcal I,\mathscr F)$-convergence. This is impossible because it would imply $\operatorname{diam}(F_\lambda)\ge \|k\eta-k\eta'\|>D$. 
\end{enumerate}

\subsection*{Declarations} 
The author confirms that no data were used in preparing this manuscript, declares that there is no conflict of interest, and received no funding for this work.

\bibliographystyle{amsplain}

\end{document}